\documentclass[11pt]{amsart}
\usepackage{amssymb,amsmath,epsfig,mathrsfs, enumerate}
\usepackage{graphicx}
\usepackage[normalem]{ulem}
\usepackage{fancyhdr}
\usepackage[margin=3cm]{geometry}
\usepackage{hyperref}
\hypersetup{
 colorlinks   = true,
 urlcolor     = blue,
 linkcolor    = blue,
 citecolor   = red ,
 bookmarksopen=true
}

\newtheorem{thm}{Theorem}[section]
\newtheorem{cor}[thm]{Corollary}
\newtheorem{lem}[thm]{Lemma}

\newtheorem{rem}[thm]{Remark}
\theoremstyle{example}

\numberwithin{equation}{section}

\newcommand\myeq{\mathrel{\stackrel{\makebox[0pt]{\mbox{\normalfont\tiny def}}}{=}}}

\newcommand{\R}{\mathbb{R}}
\newcommand{\N}{\mathbb{N}}

\newcommand{\ve}{\epsilon}

\begin{document}
\title[]
 {$C^{1, \alpha}$ Regularity for  degenerate fully nonlinear  elliptic equations with  Neumann boundary conditions}

 \author{Agnid Banerjee}
\address{Tata Institute of Fundamental Research\\
Centre For Applicable Mathematics \\ Bangalore-560065, India}\email[Agnid Banerjee]{agnidban@gmail.com}

\author{Ram Baran Verma}
\address{Tata Institute of Fundamental Research\\
Centre For Applicable Mathematics \\ Bangalore-560065, India}\email[Ram Baran Verma]{rambv88@gmail.com}

 \thanks{First author is supported in part by SERB Matrix grant MTR/2018/000267}

\begin{abstract}
In this paper, we establish $C^{1, \alpha}$ regularity  upto the boundary  for  a class of degenerate fully nonlinear elliptic equations with Neumann boundary conditions. Our   main result Theorem \ref{main} constitutes the boundary analogue of the   interior $C^{1, \alpha}$ regularity result established in \cite{IS1} for equations with similar structural assumptions.   The proof of our main result is achieved  via  compactness arguments   combined with new   boundary H\"{o}lder estimates for   equations which are uniformly elliptic when the gradient is either small or large.
\end{abstract}

\maketitle

\tableofcontents

\section{Introduction}
In this paper, we are concerned with the regularity upto the boundary  for  solutions to fully nonlinear equations of the type
\begin{equation}\label{ei}
|Du|^{\beta} F(D^2u, x) =f,
\end{equation}
 with Neumann boundary conditions, where  $\beta \geq 0$, $F$ is uniformly elliptic and $F(0,x)=0$.  Equation \eqref{ei} constitutes a subfamily of  a class of  nonlinear elliptic equations studied in a series of papers by Birindelli and Demengel  starting with \cite{BD}. We note that such equations are not uniformly elliptic, they are either degenerate or singular depending on whether $\beta>0$ or $\beta <0$.  In the singular case ( i.e. when $\beta<0$), the authors in \cite{BD} proved many important results like comparison principles and Liouville type properties. See also  \cite{BD1} for regularity results in this case.
\medskip

In the degenerate case (i.e. when $\beta>0$), the first breakthrough was made by Imbert and Silvestre in \cite{IS1} where the authors proved the interior $C^{1, \alpha}$ regularity for solutions to such equations as in \eqref{ei}.  A fairly simple  example as in \cite{IS1} shows that solutions to such equations cannot be more regular than $C^{1, \alpha}$ even when $F(D^2 u)=\Delta u$.    Subsequently, optimal $C^{1, \alpha}$ regularity results in case of concave $F$  have been obtained in the recent  interesting work \cite{ART}.  We  note that the proof of the $C^{1, \alpha}$ result  in \cite{IS1} is based on successful adaptation of  compactness arguments inspired by the ideas  as in the fundamental work of Caffarelli in \cite{Ca}. 
  We also refer the reader to the paper  \cite{BD2} for $C^{1, \alpha}$ results in case Dirichlet boundary conditions.  Our main result Theorem \ref{main} below thus  complements the regularity results previously obtained  in  \cite{IS1} and \cite{BD2}.

\medskip

Now, in order to put things in the right perspective, we note that getting a $C^{1,\alpha}$ regularity result  in general   amounts to show that the graph of the solution $u$ can be touched by an affine function so that the error is of order $r^{1+\alpha}$ in a ball of radius $r$ for every $r$ small enough. The proof of this is based on iterative argument where one ensures improvement of flatness at  every successive scale.  At each step, via rescaling, it reduces to show that if $<p,x> +u$ solves \eqref{ei} in $B_1$, then the oscillation of $u$ is strictly smaller in a smaller ball upto a linear function. This is accomplished via compactness arguments which crucially relies on apriori estimates. Now for a $u$ which solves \eqref{ei}, we have that $u-<p,x>$ is a solution to
\begin{equation}\label{it}
|Dv +p|^\beta F(D^2v, x)=f.
\end{equation}
Therefore, in order to make such a compactness argument work for $\beta>0$, it is important to get equicontinuous estimates for  equations of the type \eqref{it} independent of $|p|$. This is precisely done in \cite{IS1} using H\"older estimates for small slopes (i.e. when $|p|$ is small) established  by the same authors in their previous work  \cite{IS}  combined with a new Lipschitz estimate for large slopes which they obtain by adapting the Ishii-Lion's approach as in \cite{IL} to their setting.

\medskip

In this paper, we follow a strategy similar to that in \cite{IS1} with appropriate  adaptations. For small slopes, we establish analogous boundary H\"older estimates as in \cite{IS} for Neumann conditions  by the method of sliding cusps introduced in the same paper \cite{IS}. However for large slopes,   we could not find a suitable   adaptation of the Ishii-Lion's approach  in our setting for  getting equicontinuous estimates.  We note that  although such an  approach   has been implemented for   global oblique derivative problems  by Barles in \cite{ba},  nevertheless a suitable  localization of such an approach in  case of non-homogeneous boundary conditions is not clear to us.  Therefore, in order to overcome such an obstruction, we employ the method  of Savin as in \cite{sa}  based on sliding paraboloids  in order  to obtain  equicontinuity estimates for large slopes.  More precisely, we adapt a certain quantitative version of Savin's method  due to  Colombo and Figalli in \cite{CF}.  We also note that such oscillation  estimates are in fact established for more general fully nonlinear operators ( with structural assumptions as in SC1)-SC3) in Section \ref{lga}) and we believe that this aspect  could  possibly be of independent interest and may find other applications.  Finally for a historical account, we note that the method of sliding paraboloids seems to have originated first in a slightly different context  in the work of Cabre in \cite{Cb}.

\medskip
 As the reader will observe, the implementation of either of these approaches for Neumann boundary conditions  is somewhat delicate.  For instance in the case of small slopes, because of certain technical obstructions, our proof of the $L^{\ve}$ estimate as in Theorem \ref{bounepsilon} is based  on  the Calderon-Zygmund decomposition instead of the growing ink spot lemma  as used  in \cite{IS}. Moreover for large slopes, unlike that in \cite{IS1},  since our oscillation estimate as stated in Theorem \ref{rescaled} below only holds at large enough scales, therefore the compactness arguments in our setting required  some appropriate modifications. 
 \medskip

The paper is organized as follows.  In Section \ref{pl}, we  introduce basic notations and then  state our main result.  In Section \ref{sma}, we establish uniform boundary H\"older estimates for small slopes by the  method of sliding cusps. In Section \ref{lga}, we obtain analogous equicontinuous estimates for large slopes via sliding paraboloids. In Section \ref{end}, we finally prove our main result Theorem \ref{main} using   the compactness method which  crucially relies on the  regularity estimates proved  in Sections \ref{sma} and \ref{lga}.  Finally we refer the reader to  \cite{P} for Lipschitz regularity results for equations of the type \eqref{ei} in the singular case with homogeneous Neumann conditions.  
\medskip

In closing, we would like to mention that it remains to be seen whether similar regularity results can be obtained for more general oblique derivative conditions. This is an interesting aspect to which we would like to come back in a future study. Finally we would like the reader to note that  Neumann regularity results are also useful in the context of Signorini type obstacle problems. See for instance  \cite{ACS}, \cite{AU88}, \cite{Caf},  \cite{MS1}   and \cite{U85} to name a few.

\section{Notations and the statement of the main result}\label{pl}
For a given $r>0$ and $x\in\R^{n},$ we denote by $B_r(x)$ the ball of radius $r$ centered at $x=(x', x_n)$ and the  set $B_r(x) \cap \{ y:y_n>0\}$ by $B_r^{+}(x)$. When $x=0,$ we will occasionally denote such sets by $B_r$ and $B_r^+$ respectively. Also  the set $\{x_n=0\} \cap B_r$ will be denoted by $B_r^{0}$. Likewise $Q_r(x)$ will denote a cube of length $r$ centered at $x.$ In particular, if $x=0,$ we will use the simpler notation $Q_r$ for such a set.  $Q_r^{0}$ will refer to the set  $Q_r\cap\{y_n=0\}$.  For $x_{0}\in\{y_n=0\},$ we  also define the upper half cube of side length $r$ as follows:
\[Q^{+}_r(x_{0})=\{x\in\R^{n}~~|~~|x'-x'_0|_\infty<\frac{r}{2}~~~\text{and}~~0<x_n<r\},\]
Finally $S(n)$ will denote the set of all $n \times n$ real symmetric matrices.

\medskip

Now we list our basic structural assumptions.  We will assume  that $F$ as in \eqref{ei} is uniformly elliptic with ellipticity bounds $\lambda$ and $\Lambda$, i.e.
\begin{equation}\label{F}
\lambda ||N^+||- \Lambda ||N^-|| + F(M, x)\leq F(M+N,x) \leq F(M,x ) + \Lambda ||N^+|| - \lambda ||N^-||,
\end{equation}
where $N^+$ and $N^-$ denote  the positive and negative parts of a symmetric matrix $N$ respectively.  Moreover, we will also assume that
\begin{equation}\label{F1}
 |F(M, x) - F(M, y)| \leq \omega(|x-y|) \|M\|,
 \end{equation}
for some modulus of continuity $\omega$.
We now  state our main result.
\subsection*{Statement of the main result}
\begin{thm}\label{main}
Let $u$ be a viscosity solution to the following Neumann problem
\begin{equation}\label{deg}
\begin{cases}
|Du|^{\beta} F(D^2u, x) = f, \ \text{in $\Omega \cap B_1(0)$, $0 \in \partial \Omega$, $\beta \geq 0$},
\\
u_{\nu}=g,\ \text{on $\partial \Omega \cap B_1(0)$},
\end{cases}
\end{equation}
where $F$ satisfies the  structural assumptions in \eqref{F} and \eqref{F1}, $\Omega$ is a  $C^{2}$ domain,  $f \in C(\overline{\Omega})$ and $g \in  C^{\alpha_0}(\partial \Omega)$ for some $\alpha_0 >0$.   Then we have that  $u \in C^{1, \alpha}(\overline{  \Omega \cap B_{1/2}(0)})$ for some $\alpha > 0$ depending on $n, \lambda, \Lambda, \omega, \beta,\alpha_0$ and the $C^{2}$ character of $\Omega$. Here $\nu$ denotes the outward unit normal to $\Omega$.
\end{thm}

\begin{rem}
For the precise notion of viscosity solutions to fully nonlinear Neumann problems, we refer the reader to \cite{LZ}.
\end{rem}

From  Theorem \ref{main}, the following corollary  can be deduced. 

\begin{cor}\label{main1}
Let $u$ be a viscosity solution to the following Robin boundary problem
\begin{equation}\label{degb}
\begin{cases}
|Du|^{\beta} F(D^2u, x) = f, \ \text{in $\Omega \cap B_1(0)$, $0 \in \partial \Omega$, $\beta \geq 0$},
\\
u_{\nu} + h(x) u=g,\ \text{on $\partial \Omega \cap B_1(0)$},
\end{cases}
\end{equation}
where $F$ satisfies the assumptions in \eqref{F} and \eqref{F1}, $\Omega \in C^{2}$,  $f \in C(\overline{\Omega})$ and $h, g\in  C^{\alpha_0}(\partial \Omega)$ for some $\alpha_0 >0$.   Then $u \in C^{1, \alpha}(\overline{ B_{1/2} \cap \Omega})$ for some $\alpha > 0$ depending on $n, \lambda, \Lambda, \omega, \beta,\alpha_0$ and the $C^{2}$ character of $\Omega$.
\end{cor}

\section{H\"older estimates upto the boundary for equations which are uniformly elliptic when the gradient is large}\label{sma}
In this section we establish uniform  non-perturbative H\"older estimates for equations of the type \eqref{it} for small $|p|'s$ (say when $|p| \leq a_0$ for some $a_0>0$). We first note that this in turn is equivalent to getting similar estimates for 
We first note that establishing uniform H\"older estimates for small $|p|$ ( say $|p| \leq a_0$)  upto the boundary for equations of the type
\[
F(D^2u, x)= \frac{f}{|Du + p|^{\beta}}
\]
which lends itself an uniformly elliptic structure when   say $|Du|$ satisfies  $|Du| > 2a_0 +1$ in the viscosity sense. Therefore,  this reduces  to getting  uniform  H\"older estimates for equations which are uniformly elliptic when the gradient is large. We thus introduce the relevant framework similar to that in \cite{IS}.

For a given $\gamma>0$ and   $0< \lambda< \Lambda$,   let   $\mathcal{P}^{\pm}_{\lambda, \Lambda, \gamma}$ be defined by
\begin{equation}\label{def}
 \mathcal{P}_{\lambda, \Lambda, \gamma}^{+}(D^{2}v, Dv)=
\begin{cases}
    \Lambda\text{tr} D^{2}v^{+}-\lambda\text{tr} D^{2}v^{-}+\Lambda|Dv|,\,& \text{if}~|Dv|\geq \gamma\\
    +\infty,              & \text{otherwise}
\end{cases}
\end{equation}
and
\begin{equation}\label{def}
 \mathcal{P}_{\lambda, \Lambda, \gamma}^{-}(D^{2}v, Dv)=
\begin{cases}
    \lambda\text{tr} D^{2}v^{+}-\Lambda\text{tr} D^{2}v^{-} - \Lambda|Dv|,\,& \text{if}~|Dv|\geq \gamma\\
    -\infty,              & \text{otherwise}.
\end{cases}
\end{equation}

When the context is clear, we will frequently  denote $\mathcal{P}^{\pm}_{\lambda, \Lambda, \gamma}$ simply by $\mathcal{P}^{\pm}$.  We first recall the interior $C^{\alpha}$ estimate as established in Theorem 1.1 in \cite{IS}. 
\begin{thm}\label{interho}
For any continuous function $u:\overline{B_{1}}\longrightarrow\R,$  satisfying in the viscosity sense,
\begin{equation}\label{inholder}
\left\{
\begin{aligned}{}\
\mathcal{P}^{-}(D^{2}u,Du)&\leq C_{0}~\text{in}~B_{1},\\
\mathcal{P}^{+}(D^{2}u,Du)&\geq-C_{0}~\text{in}~B_{1},\\
\|u\|_{L^{\infty}(B_{1})}&\leq C_{0},
\end{aligned}
\right.
\end{equation}
we have that  $u\in C^{\alpha}(B_{\frac{1}{2}}(0))$ for some $\alpha$ depending on $\lambda, \Lambda$ and the dimension $n$. Furthermore, the following estimate holds, 
\begin{equation}
\|u\|_{C^{\alpha}(B_{\frac{1}{2}}(0))}\leq C(n,\lambda,\Lambda,\gamma, C_{0}).
\end{equation}
\end{thm}

\begin{rem}\label{scaling}
It is clear from the definition of $\mathcal{P}^{\pm}(M, p)$ that if $u$ satisfies $\mathcal{P}^{+}(D^{2}u, Du)\geq L$ $\big(\text{resp.}, \mathcal{P}^{-}(D^{2}u, Du)\leq L\big)$ in $\Omega,$ then the rescaled function $v(x)=Mu(x_{0}+rx)$ satisfies
\[\mathcal{P}_{r,M}^{+}(D^{2}v, Dv)\geq Mr^{2}L~~\big(\text{resp.}~~~\mathcal{P}_{r,M}^{-}(D^{2}v, Dv)\leq MLr^{2}\big)~~\text{in}~\frac{1}{r}\Omega - x_0,\]
where
\[
 \mathcal{P}_{r,M}^{+}(D^{2}v, Dv)=
\begin{cases}
    \Lambda\text{tr} D^{2}v^{+}-\lambda\text{tr} D^{2}v^{-}+r\Lambda|Dv|,\,& \text{if}~|Dv|\geq rM\gamma\\
    +\infty,              & \text{otherwise}.
\end{cases}
\]

Similarly, $\mathcal{P}_{r,M}^{-}(D^{2}v, Dv)$ is also defined.
\end{rem}




We  now  proceed with our proof of analogous boundary estimates. In Sections \ref{sma} and  \ref{lga}, we only restrict to the case when $\partial \Omega= \{x_n=0\}$. In Section \ref{end}, we then  show how to reduce to flat boundary conditions.  The following result is the measure to uniform estimate at the boundary, which is analogue to  Lemma 3.1 in \cite{IS}.

\begin{thm}\label{t1a}
There exist two small constants $\epsilon_{0}>0$ and $\delta>0,$ and a large constant $K>0,$ such that if $\gamma \leq \ve_0$, then for any lower semicontinuous function $u:\overline{Q^{+}_{1}}\longrightarrow\R$ satisfying
\begin{equation}\label{m12*}
\left\{
\begin{aligned}{}
&u\geq0~~\text{in}~~Q^{+}_{1},\\
\mathcal{P}^{-}(D^{2}u, Du)&\leq1~\text{in}~Q^{+}_{1},\\
u_{x_n}& \leq 0~\text{on}~Q_{1}^{0},\\
|\{u>K\}\cap Q^{+}_{1}|&\geq(1-\delta)|Q^{+}_{1}|,
\end{aligned}
\right.
\end{equation}
we have that $u>1$ in $Q^{+}_{\frac{1}{16\sqrt{n}}}.$
\end{thm}

\begin{proof}

The proof  is divided into three steps.

\medskip

\emph{Step 1:}
Similar to  that in \cite{IS}, we  first assume  that $u$ is a classical solution of \eqref{m12*}, i.e.  let  $u\in C^{2}(Q^{+}_{1}) \cap C(\overline{Q^{+}_{1}})$ and satisfies the Neumann condition in the viscosity sense.  Suppose  on the contrary that for all $\epsilon_{0}, \delta$ and $K$ such that for which \eqref{m12*} holds, there exists $x_{0}\in Q^{+}_{\frac{1}{16\sqrt{n}}}$ such that $u(x_{0})\leq1.$ Let us consider the following set
$G=\{u>K\}\cap  Q^{+}_{\frac{1}{16\sqrt{n}}}$. Given $x\in G$, let $y\in\overline{Q^{+}_{1}}$ be a point such that
\begin{equation}\label{e0}
u(y)+10|y-x|^{1/2}=\min_{\overline{Q^{+}_{1}}}\{u(z)+ 10 |z-x|^{1/2}\}
\end{equation}
i.e. we slide the cusp  with vertex at $x$  till touches the graph of $u$ for the first time. 
Now on  one hand, since $u\geq0$ in $\overline{Q^{+}_{1}}$ and $x\in G\subset B^{+}_{1/8},$ therefore we have
\begin{align}
u(\xi)+10|\xi-x|^{\frac{1}{2}}&\geq10|\xi-x|^{\frac{1}{2}}\\
&>5\sqrt{\frac{7}{2}},
\end{align}
for any $\xi\in\partial Q^{+}_{1}\cap\{x_{n}>0\}.$
On the other hand,
\begin{align}
u(x_{0})+10|x_{0}-x|^{\frac{1}{2}}&\leq1+10\times(\frac{1}{4})^{\frac{1}{2}}\\
&=6<5\sqrt{\frac{7}{2}}.
\end{align}
This shows that $y \notin \partial Q_{1}^{+} \cap \{x_n>0\}$. We now show that  $y\not\in Q_{1}^{0}.$  If that is not the case, then since $u_{x_n} \leq 0$ in the viscosity sense, therefore necessarily  we  must have
\begin{align}\label{phi}
\limsup_{h\rightarrow0}\Big[\frac{\phi(y',he_{n})-\phi(y',0)}{h}\Big]\leq 0
\end{align}
where $\phi(\cdot)=-10|\cdot-x|^{\frac{1}{2}}.$ However a direct calculation shows that the quantity in \eqref{phi} equals
\begin{align*}
-5\frac{(y-x)}{|y-x|^{\frac{3}{2}}}\cdot e_{n}
=\frac{5x_{n}}{|y-x|^{\frac{3}{2}}}>0~~\hspace{0.5cm}~(\text{since}~~y_{n}=0),
\end{align*}
which is a contradiction to \eqref{phi}. Therefore, the minimum will never be achieved on the boundary and thus $y\in Q^{+}_{1}.$ At this point, the rest of the proof is similar to that in the in interior case ( see Proposition 3.3 in \cite{IS}) but we  nevertheless provide the details for the sake of completeness.

\medskip

Let  $K=1+5\sqrt{\frac{7}{2}}.$ In this way, we can ensure  that $u(y)<K.$ In particular $x\not=y$ and therefore $|z-x|^{\frac{1}{2}}$ is differentiable at $z=y.$ Note that for one value of $x,$ there can be more than one  $y$ where the minimum is achieved. However, the value of $y$ determines $x$ completely since we must have \[Du(y)=5(x-y)|y-x|^{-\frac{3}{2}}.\]
Let us now set $\psi(\xi)=-10|\xi|^{\frac{1}{2}}.$ Then from the extrema conditions, we have 
\begin{align}
Du(y)&=D\psi(y-x),\label{1}\\
D^{2}u(y)&\geq D^{2}\psi(y-x)\label{1a}.
\end{align}
The relations \eqref{1} and \eqref{1a}, together with $\mathcal{P}^{-}(D^{2}u, Du)\leq1,$ imply that
\begin{equation}\label{1b}
|D^{2}u(y)|\leq C(1+|D^{2}\psi(y-x)|+|D\psi(y-x)|),
\end{equation}
 as long as $\epsilon_{0}\leq \min_{B_{\sqrt{n}}}|D\psi|$.  Note that over here, $C$ only depends on the ellipticity constants and the dimension. Since for each value of $y,$ there is only one value of $x,$ so we can define a map $\tau(y):=x.$ Let  $U$ be the domain of $\tau.$ It is clear that $U\subset\{z~:~u(z)<K\}$ and $\tau(U)=G.$\\
By putting $x=\tau(y)$ in \eqref{1} and employing the chain rule, we get
\[D^{2}u(y)=D^{2}\psi(y-\tau(y))(I-D\tau(y)).\]
Solving for $D\tau$ and using the estimate \eqref{1b}, we get
\begin{equation}\label{compr}
|D\tau(y)|\leq1+C\frac{1+|D^{2}\psi(y-x)|+|D\psi(y-x)|}{|D^{2}\psi(y-x)|}\leq \tilde C.\end{equation}
The reader should note over here in \eqref{compr},  we crucially used the fact that all the eigenvalues of $D^2 \psi$ are comparable. 
Now, since
\[ 
\frac{\left|Q^{+}_{\frac{1}{16\sqrt{n}}} \right|}{|Q^{+}_{1}|} \geq c(n),\]  therefore  in view of the last condition in  \eqref{m12*}  and the fact that $U\subset\{z~|~u(z)<K\},$ we obtain
\[(1-c\delta)| Q^{+}_{\frac{1}{16\sqrt{n}}}|\leq |G|=\int_{U}|\text{Det}\ \tau(y)|dy\leq C|U| \leq C\delta|Q^{+}_{1}|.\]
 This is a contradiction if $\delta$ is small enough. This completes the proof of \emph{Step 1}.

\medskip

\emph{Step 2:} Assuming that the  Theorem  \ref{t1a}  holds for semiconcave supersolutions, we  now  show that this in turn  implies that the conclusion  remains true for  lower semi-continuous supersolution $u.$

Let $u$ be a merely lower semi continuous supersolution defined in $\overline{Q^{+}_{1}}.$ Let $v:=\min\{u,2K\},$ where $K$ is as in \emph{Step 1}. Note that $v$ is still a supersolution because it is the minimum of two supersolutions. Indeed, suppose that $v-\phi$ has minimum at $x_{0}.$ There are two possibilities:

\begin{equation}
\begin{cases}
\text{ 1)}\quad x_{0}\in Q^{+}_{1}\ \text{or}
\\
\text{ 2)}\quad x_{0}\in Q_{1}^{0}.
\end{cases}
\end{equation}
 We first note that there  two possible subcases under the Case 1).\\(1a)~If $v(x_{0})=u(x_{0})$, then we have
\[u(x_{0})-\phi(x_{0})=v(x_{0})-\phi(x_{0})\leq v(x)-\phi(x)\leq u(x)-\phi(x).\]

In this case, the desired  differential inequality is seen to be valid for $\phi$ because $u$ satisfies such an inequality in the viscosity sense. 

\medskip

(1b) Suppose instead  that $v(x_{0})=2K$, then we have
\[2K-\phi(x_{0})=v(x_{0})-\phi(x_{0})\leq v(x)-\phi(x)\leq 2K-\phi(x)\]
and conclusion in this case   follows  from the extrema conditions   for $\phi$.  Similarly the Neumann condition when  Case 2) holds is seen to be satisfied. 

\medskip

As in \cite{IS}, for a given  $\delta>0,$ we now  consider the inf-convolution of $v$ defined as follows:
\[v_{\epsilon}(x)=\inf_{y\in~\overline{Q^{+}_{1-\widetilde{\delta}}}}\Big(v(y)+\frac{1}{2\epsilon}|y-x|^{2}\Big),\]
where $\tilde{\delta}=\delta/2.$ For any $x\in Q^{+}_{1},$ using the fact that $v_{x_n} \leq 0$, it follows in a standard way that the infimum above  will be achieved at any point $y_{0}\in \overline{Q^{+}_{1-\tilde{\delta}}}\setminus Q_{1}^{0}$. See for instance the proof of  Lemma 5.2 in \cite{MS}.

We now make the following claim. 

\medskip

\textbf{Claim:}~ For any $\epsilon>0$ satisfying $2\sqrt{2K\epsilon}<\delta/4,$ $v_{\epsilon}$ is supersolution to the following problem
\begin{equation}\label{ap}
\left\{
\begin{aligned}{}
\mathcal{P}^{-}(D^{2}v_{\epsilon}, Dv_{\epsilon})&\leq1~\text{in}~Q^{+}_{1-\delta},\\
(v_{\epsilon})_{x_n}&\leq0~\text{on}~Q_{1-\delta}^{0}.
\end{aligned}
\right.
\end{equation}

The proof of this claim follows exactly the same way as that of Lemma 5.3 in \cite{MS} and so we skip the details.  Then by noting that  $v_{\epsilon}$ is semiconcave and satisfies \eqref{ap}, we can now apply the conclusion of \emph{Step 1} to $v_{\ve}$ and then  by a  limiting argument as in the proof of Proposition 3.4 in \cite{IS}, we thus conclude that the assertion in \emph{Step 2 } holds.

\medskip

\emph{Step 3:}
Finally the fact that the  conclusion of Theorem \ref{t1a} holds when $u$ is a semiconcave viscosity supersolution of \eqref{m12*} follows by repeating the interior arguments as in the proof of Proposition 3.5 in \cite{IS}. 
Note that the Neumann condition $u_{x_n} \leq 0$ ensures that  as in \emph{Step 1} that the minimum in \eqref{e0} is attained on the set $Q_{1}^{+} \setminus \{x_n=0\}$. This finishes the proof.\end{proof}
\subsection*{Barrier function and doubling type lemma}
As mentioned in the introduction,  since our proof of the $L^{\ve}$ estimate  relies on Calderon-Zygmund decomposition instead of the  growing Ink-spot lemma  as employed in \cite{IS} because of certain technical obstructions, therefore  we need a somewhat adjusted doubling type lemma  as stated in Theorem \ref{double} below.

\medskip

Similar to \cite{IS}, we consider the function 
\[
V(x)=|x|^{-\sigma}+\epsilon_{n,\sigma} x_{n}=h(x)+\epsilon_{n,\sigma} x_{n},
\]
 where $\epsilon_{n,\sigma}>0$ is a positive constant depending on $\sigma$ and $n$  and  will be subsequently chosen.  We let $r= |x|$. As the reader will see, unlike  the interior case as in \cite{IS}, this  additional term $\ve_{n, \sigma} x_n$  accounts for the adjustment required due to the presence of the  Neumann  condition.  Using $D^2V=D^2h$ and also the fact that  $h$ is radial, we can  assert that  the eigenvalues of $D^{2}h(x),$ for $x\not=0,$ are $-\sigma r^{-\sigma-2}$ with multiplicity $n-1$ and $\sigma(\sigma+1)r^{-\sigma-2}$ with multiplicity 1.
Therefore, for $x\not=0,$ we have
\begin{equation}
\mathcal{P}^{-}(D^{2}V(x), DV(x))=\lambda\sigma(\sigma+1)r^{-\sigma-2}-\Lambda(n-1)\sigma r^{-\sigma-2}-\Lambda|\epsilon_{n,\sigma}e_{n}-\sigma r^{-\sigma-2}x|,
\end{equation}
as long as $|DV(x)|\geq\gamma.$ A standard calculation shows
\begin{equation}\label{lar17}
\begin{aligned}{}
\mathcal{P}^{-}(D^{2}V(x), DV(x))&=\lambda\sigma(\sigma+1)r^{-\sigma-2}-\Lambda(n-1)\sigma r^{-\sigma-2}-\Lambda|\epsilon_{\sigma}e_{n}-\sigma r^{-\sigma-2}x|\\
&=\sigma r^{-\sigma-2}\Big(\lambda(\sigma+1)-\Lambda(n-1)-\Lambda|(\frac{\epsilon_{n,\sigma}}{\sigma})r^{\sigma+2}e_{n}-x|\Big)\\
&\geq\sigma r^{-\sigma-2}\Big(\lambda(\sigma+1)-\Lambda(n-1)-\frac{\Lambda\epsilon_{n,\sigma}r^{\sigma+2}}{\sigma}-\Lambda r\Big).
\end{aligned}
\end{equation}
The next lemma corresponds to the spread of the positivity set needed to apply the  Calderon-Zygmund type lemma in the upper half space.  
\begin{thm}\label{double}
There exists an $\epsilon_{0}>0$ depending on the ellipticity and dimension such that if  $\gamma \leq \epsilon_0$, $u:Q^{+}_{8n}\longrightarrow\R,$
satisfies
\begin{equation}\label{m12}
\left\{
\begin{aligned}{}
&u\geq0~~\text{in}~~Q^{+}_{8n},\\
\mathcal{P}^{-}_{\lambda, \Lambda, \gamma} (D^{2}u, Du)&\leq1~\text{in}~Q^{+}_{8n},\\
u_{x_n}&\leq0~\text{on}~Q_{8n}^0,\\
\end{aligned}
\right.
\end{equation}
 and $u>K$ on $Q^{+}_{\frac{1}{16\sqrt{n}}}$ for a sufficiently large $K$(depending on $\Lambda,\lambda,n,\gamma$), then $u>1$ in $Q^{+}_{3}.$
\end{thm}
\begin{proof}
We first  observe that
\[B^{+}_{\frac{1}{32\sqrt{n}}}\subset Q^{+}_{\frac{1}{16\sqrt{n}}}~~\text{and}~~Q^{+}_{3}\subset B^{+}_{3\sqrt{n}}\subset B^{+}_{4n}\subset Q^{+}_{8n}.\]
Then we  consider the following barrier function:
\begin{equation}
B(x)=\frac{K}{2[32\sqrt{n}]^{\sigma}}\Big[|x|^{-\sigma}-(4n)^{-\sigma}+\epsilon_{n,\sigma}(x_{n}-8n\sqrt{n})\Big],
\end{equation}
with $\epsilon_{n,\sigma}=(128n)^{-8(\sigma+2)}$.\\
For any value of $K\geq 1,$ we note that  $B$ has the following properties:
\begin{enumerate}
\item{}~$B(x)\leq0$~~~for any $|x|\geq4n.$
\item{} $B(x)\leq \frac{K}{2}<K$ for any $x\in \partial B_{\frac{1}{32\sqrt{n}}}.$ In particular for any $x\in\partial B_{\frac{1}{32\sqrt{n}}}\cap\{x_{n}>0\},$ $B(x)<K.$
\item{} For any $x\not=0,$ and $x_{n}=0,$ $\frac{\partial B}{\partial x_n}(x)=\frac{\epsilon_{n,\sigma}K}{2[32\sqrt{n}]^{\sigma}}.$ In particular, 
\begin{equation}\label{bn}
\frac{\partial B}{\partial x_n}(x)>0
\end{equation}
 for  $x\not=0$.
\end{enumerate}
We now choose $\sigma$ sufficiently large such that the following holds:
\begin{equation}
\lambda(\sigma+1)-\Lambda(n-1)-\frac{\Lambda(8n\sqrt{n})^{\sigma+2}}{\sigma(128n)^{8(\sigma+2)}}-\Lambda(8n\sqrt{n})\geq2.
\end{equation}
Having chosen $\sigma,$ it is always possible to choose $K\geq1$~(sufficiently large), such that following inequalities hold:
\begin{enumerate}
\item{}~$|DB(x)|\geq\gamma$ in $Q^{+}_{8n},$
\item{}~~$|B(x)|>1$ in $B^{+}_{3\sqrt{n}},$
\item{}~~$\mathcal{P}^{-}(D^{2}B, DB)\geq2~\text{in}~Q^{+}_{8n}.$
\end{enumerate}
Now, we  claim  that $u\geq B$ in $(B_{4n}\setminus B_{\frac{1}{16\sqrt{n}}})\cap\{x_{n}\geq0\}.$  If not, then there exists an $z_{0}\in(B_{4n}\setminus B_{\frac{1}{16\sqrt{n}}})\cap\{x_{n}\geq0\}$ which corresponds to a  negative minimum of $u-B$ in that same set. Then there are two possibilities
\begin{enumerate}
\item{}~~if $(z_{0})_{n}=0$, then  we  must have $\frac{\partial  B}{\partial x_n} \leq 0$ which in view of \eqref{bn} above is not possible.
\item{}~~if $z_{0}$ is an interior point, then we  have that  $ \mathcal{P}^{-}(D^{2}B(z_{0}), DB(z_{0}))\geq2,$ which is again a contradiction.
\end{enumerate}
This proves the claim. Therefore  for $\epsilon=\min_{B^{+}_{\frac{1}{16\sqrt{n}}}}(u/K-1),$ we  obtain $u\geq(1+\epsilon)K>1$ in $B^{+}_{3\sqrt{n}}.$
\end{proof}
As a consequence, we have the following corollary which is the key ingredient in our proof of $L^{\epsilon}$ estimate.
\begin{cor}\label{dli}
There exist small constants $\epsilon_{0}>0$ and $\delta>0$ and a large constant $K>0,$ such that if $\gamma \leq\epsilon_{0},$ then for any  lower semicontinuous function $u:\overline{Q^{+}_{8n}}\longrightarrow\R,$ satisfying
\begin{equation}\label{m121}
\left\{
\begin{aligned}{}
&u\geq0~~\text{in}~~Q^{+}_{8n},\\
\mathcal{P}^{-}(D^{2}u, Du)&\leq1~\text{in}~Q^{+}_{8n},\\
u_{x_n}&\leq0~\text{on}~Q_{8n}^0,\\
|\{u>K\}\cap Q^{+}_{1}|&>(1-\delta)|Q^{+}_{1}|,
\end{aligned}
\right.
\end{equation}
we have $u>1$ in $Q^{+}_{3}.$
\end{cor}
\begin{proof}
Let $K_{1}$ and $K_{2}$ be the (renamed) constants from Theorems \ref{t1a}~and~\ref{double} respectively. We claim that  $K$ can be taken to be  $ K_1 K_2$.  With such a choice of $K$,  we note that the function $v=u/K_{2}$ satisfies the assumption of Theorem \ref{t1a}. From there  we conclude that $v>1$ in $Q^{+}_{\frac{1}{16\sqrt{n}}},$ i.e, $u>K_{2}$ in $Q^{+}_{\frac{1}{16\sqrt{n}}}.$ Now we can apply the  doubling  result Theorem \ref{double} to finally obtain that  $u>1$ in $Q^{+}_{3}.$
\end{proof}

We now state and prove  the  boundary version of the $L^{\ve}$ estimate.
\begin{thm}\label{bounepsilon}
There exists a small enough $\ve,  \ve_0>0$ such that  if $\gamma \leq \ve_0$, then  for any $u$ satisfying
\begin{equation}\label{m12121}
\left\{
\begin{aligned}{}
&u\geq0~~\text{in}~~Q^{+}_{8n},\\
&\mathcal{P}^{-}(D^{2}u, Du)\leq1~\text{in}~Q^{+}_{8n},\\
&u_{x_n}\leq0~\text{on}~Q_{8n}^0,\\
&\inf\limits_{Q^{+}_{3}} u \leq1,
\end{aligned}
\right.
\end{equation}
we have
\begin{align}
|\{u>t\}\cap Q^{+}_{1}|&\leq \tilde{C}t^{-\epsilon},\quad t>0.\label{lin5}
\end{align}
\end{thm}
\begin{proof}
In order to prove \eqref{lin5}, note that it suffices to show that  for  $\delta>0$ as in Corollary \ref{dli},
\begin{equation}\label{lin6}
|\{u>(C_{0}K)^{m}\}\cap Q^{+}_{1}|\leq  (1-\delta/2)^{m} |Q_{1}^+|
\end{equation}
for $K$ as in Corollary \ref{dli} and $C_0$  sufficiently large   which will be  chosen below. For $m=1,$ since $\inf\limits_{Q^{+}_{3}} u \leq1$ so by Corollary \ref{dli}
 we find
 \[|\{u>K\}\cap Q^{+}_{1}|\leq(1-\delta)|Q^{+}_{1}|.\]

Now assume that the result is true for $m-1,$ that is,
 \begin{equation}
 |\{u>(C_{0}K)^{m-1}\}|\leq (1-\frac{\delta}{2})^{m-1}|Q^{+}_{1}|.
 \end{equation}
Let us set
 \[A_{m}=\{u>(C_{0}K)^{m}\}\cap Q^{+}_{1}~~~\text{and}~~~A_{m-1}=\{u>(C_{0}K)^{m-1}\}\cap Q^{+}_{1}.\]
 We claim that
 \begin{equation}\label{cz}
 |A_{m}| \leq (1-\delta/2) |A_{m-1}|.
 \end{equation}
If not, then by the  Calderon-Zygmund lemma applied to cubes in the upper half space, we have that there exists  a dyadic cube
$Q$  such that
\begin{equation}\label{contra}
|A_{m}\cap Q|>(1-\delta/2)|Q|
\end{equation}

and $2 Q= \tilde{Q}\not\subset A_{m-1},$ i.e. there is a point $x_{1}\in \tilde Q$ such that $u(x_{1})\leq (C_{0}K)^{m-1}.$  Let us consider the following cases:\\

\medskip
\textbf{Case ~1:}~Suppose $Q= Q_{\frac{1}{2^i}} (x_0)$ such that $|(x'_{0},(x_{0})_{n})-(x'_{0},0)|\geq \frac{4n}{2^{i}}.$ In this case, it is easy to observe that $Q_{\frac{8n}{2^{i}}} (x_0) \subset Q^{+}_{8n}.$ Therefore, the rescaled  function $\tilde{u}:Q_{8n}\longrightarrow\R,$ defined by $\tilde{u}(y)=\frac{1}{(C_{0}K)^{m-1}}u(x_{0}+\frac{e_{n}}{2^{i+1}}+\frac{y}{2^{i}})$ satisfies the following differential inequality
\begin{equation}\label{m121}
\left\{
\begin{aligned}{}
&\tilde{u}\geq0~~\text{in}~~Q_{8n},\\
\mathcal{P}^{-}(D^{2}\tilde{u}, D\tilde{u})&\leq1~\text{in}~Q_{8n},\\
\tilde{u}(y_{1})&\leq 1 ~~\text{for~~some}~~y_{1}\in~Q_{3},
\end{aligned}
\right.
\end{equation}
for a smaller $\gamma$ in view of the discussion in Remark \ref{scaling}. Therefore, we can employ the interior version of  Corollary \ref{dli} to conclude that
\begin{align}
|\{\tilde{u}>K\}&\cap Q_{1}|\leq (1-\delta/2)|Q_{1}|,
\end{align}
which in particular implies
\begin{align}
|A_m \cap Q|= |\{u>(C_{0}K)^{m}\}&\cap Q|\leq (1-\delta/2)|Q|.
\end{align}
 This contradicts \eqref{contra}.

 \medskip

\textbf{Case~2}: Now suppose instead  that either $Q=Q_{\frac{1}{2^i}} (x_0)$ or $Q= Q^{+}_{\frac{1}{2^{i}}}(x_{0})$ with  
\[|(x'_{0},(x_{0})_{n})-(x'_{0},0)|\leq \frac{4n}{2^{i}}.\]
In this case, due to the nature of the Calderon-Zygmund decomposition for cubes in the upper half space, there are two possibilities\\
$[i]~~(x_{0})_{n}=0$ or\\
$[ii]~(x_{0})_{n}\geq\frac{1}{2^{i}}.$\\
In case 2  [i], 
we again consider the rescaled function $\tilde{u}~:Q^{+}_{8n}\longrightarrow\R$ defined by
\begin{equation}
\tilde{u}(y)=\frac{1}{(C_{0}K)^{m-1}}u(x_{0}+\frac{y}{2^{i}}),
\end{equation}
which satisfies the following differential inequality
\begin{equation}\label{m321}
\left\{
\begin{aligned}{}
&\tilde{u}\geq0~~\text{in}~~Q^{+}_{8n},\\
\mathcal{P}^{-}(D^{2}\tilde{u}, D\tilde{u})&\leq1~\text{in}~Q^{+}_{8n},\\
\tilde{u}_{x_n}&\leq 0~\text{on}~Q_{8n}^0,\\
\text{and}~~\tilde{u}(z_{1})&\leq 1~~\text{for~some}~z_{1}\in Q^{+}_{3}
\end{aligned}
\right.
\end{equation}
Therefore, by  Corollary \ref{dli} we note
\begin{align}
|\{\tilde{u}>K\}&\cap Q^{+}_{1}|\leq (1-\delta/2)|Q^{+}_{1}|
\end{align}
or equivalently,
\begin{align}
|\{u>(C_{0}K)^{m}\}&\cap Q|\leq (1-\delta/2)|Q|,
\end{align}
which  contradicts \eqref{contra} as before.

\medskip

Instead if   Case 2 [(ii)] happens, i.e. say  $(x_{0})_{n}\geq \frac{1}{2^{i}}.$ Now since we also have that $(x_{0})_{n} \leq 4n/2^i$,  therefore, given $\delta_0$ such that  $0<\delta_{0}< 1$,  there exists a cube  $ Q^{\delta_{0}}\subset Q^{+}_{1}$ of size comparable to  $Q$  which contains $Q$ such that $\text{dist}(Q^{\delta_{0}},~\{x_{n}=0\})= \delta_{0}/2^i.$ We now make the following claim. 

\medskip

\textbf{Claim:}~If $C_{0}$ is large enough, then the function
\[v(y)=\frac{u(y)}{(C_{0}K)^{m-1}}>K~~\text{in} ~~Q^{\delta_{0}}.\]
Proof of the claim: Suppose  on the contrary that  there exists a point $y_{0}\in Q^{\delta_{0}}$ such that
\[v(y_{0})\leq K.\]
Then the function defined by $w(z)=\frac{v(z)}{K},$ satisfies $w(y_{0})\leq1.$ So by the interior $L^{\epsilon}$ estimate we have
\[|\{w>t\}\cap Q^{\delta_{0}}|\leq C(\epsilon, \delta_0)t^{-\epsilon} |Q^{\delta_0}|.\]
Note that such an estimate is a consequence of the interior $L^{\ve}$ estimate in \cite{IS}  followed by a standard covering argument.  We also note that the constant $C= C(\ve, \delta_0)$ can be chosen to be independent of $i$ in view of  scale invariance of the estimates ( note that the size of both $Q^{\delta_0}$ as well as $Q$ are comparable to $\frac{1}{2^i}$) , see for instance Remark \ref{scaling}.
Therefore, in particular,
\begin{align}
|\{w>C_{0}\}\cap Q^{\delta_{0}}|&\leq C(\epsilon, \delta_0)C_{0}^{-\epsilon}|Q^{\delta_0}|.\label{contra3}
\end{align}

Now we note that since
\[\{w>C_{0}\}=\{v>C_{0}K\}=\{u>(C_{0}K)^{m}\},\]

therefore this implies that the following holds,
\[|\{u>(C_{0}K)^{m}\}\cap Q^{\delta_{0}}|\leq C(\epsilon, \delta_0)C_{0}^{-\epsilon}|Q^{\delta_0}|.\] Then using \eqref{contra}, we have
\begin{equation}\label{pl}
|\{w>C_{0}\}\cap Q|=|\{u>(C_{0}K)^{m}\}\cap Q|>(1-\delta/2)|Q|.\end{equation}

Now, we choose the smallest cube $\hat{Q}^{+}$ with base at $\{x_n=0\}$ which contains  $Q^{\delta_0}$  and we also set $\tilde{C}(\delta_{0})=\frac{|Q^{\delta_{0}}|}{|\hat{Q}^+|}$.
Note that  we  have that  $\tilde{C}(\delta_{0})\rightarrow1$ as $\delta_{0}\rightarrow0.$  Thus we can  choose $\delta_{0}$ sufficiently small such that $\tilde{C}(\delta_{0})>(1-\delta),$ where $\delta$ is from Corollary \ref{dli}. We then  let $C(\delta_{0})=|Q|/|Q^{\delta_{0}}|$.  It is easy to see that $C(\delta_0)$ is bounded from below uniformly as $\delta_0 \to 0$. Therefore  we have from \eqref{pl}
\begin{equation}\label{cotra2}
|\{w>C_{0}\}\cap Q^{\delta_{0}}|>(1-\delta/2)|Q|=(1-\delta/2)C(\delta_{0})|Q^{\delta_{0}}|.
\end{equation}
At this point, if we  choose $C_{0}$ sufficiently large such that $2C(\epsilon, \delta_0)C_{0}^{-\epsilon}<(1-\delta/2)C(\delta_{0}),$ then from  \eqref{contra3}  we obtain
\[
|\{w>C_{0}\}\cap Q^{\delta_{0}}| < (1-\delta/2)C(\delta_{0})|Q^{\delta_{0}}|
\]

 which contradicts \eqref{cotra2}. This proves the claim.\\
Consequently, we have
\begin{align}
|\{v>K\}\cap \hat{Q}^+|&\geq |\{v>K\}\cap Q^{\delta_{0}}|~~~\hspace{0.5cm}\big(\text{since}~~Q^{\delta_{0}}\subset \hat{Q}^+\big)\\
&=|Q^{\delta_{0}}|~\hspace{0.8cm}~\big(\text{since}~~v>K~~\text{in}~~Q^{\delta_{0}}\big)\notag\\
&=\tilde{C}(\delta_{0})|\hat{Q}^+|\notag\\
&>(1-\delta)|\hat{Q}^+|\notag.
\end{align}
Therefore by invoking Corollary \ref{dli}, we    conclude that $v> 1$ in $3 \hat{Q}^+$  and hence $v > 1$ in $\tilde Q$ since $\tilde Q \subset 3 \hat{Q}^+$. Now given that $A_{m-1} = \{u> (C_0 K)^{m-1}\} \cap Q_1^+= \{v>1\} \cap Q_1^+$, therefore this contradicts the fact that $\tilde Q \not \subset A_{m-1}$.  The conclusion of the  Theorem thus follows.
\end{proof}

We also need the following uniform estimate as in Theorem \ref{unif} below  which is a consequence of  a   scaled version of the  above $L^{\ve}$ estimate. Such an estimate plays a crucial role  in the proof of  H\"older regularity of the solutions upto the boundary similar to that in the interior case as in \cite{IS}.  Before stating such a result, we make the following important  remark.

\begin{rem}\label{cot}
Given $\epsilon_0>0$ as  in Theorem \ref{bounepsilon},    we will choose $C_1$   large enough in the hypothesis of  Theorem \ref{unif} below  such that $\epsilon_0>\frac{2}{C_{1}\Lambda}$ where $\Lambda$ is the ellipticity upper bound.
\end{rem}
\begin{thm}\label{unif}
There exist small constants $\tilde{\epsilon}_{0}, c_0>0$ and  $\alpha, r_0 \in (0,1)$ such that if $\gamma\leq \tilde{\epsilon}_{0}$,  then for  any lower semicontinuous function $u:B^{+}_{(4n)r}:\longrightarrow\R$   satisfying the following  differential inequalities for $r \leq r_0,$
\begin{equation}\label{m321}
\left\{
\begin{aligned}{}
&u\geq0~~\text{in}~~B^{+}_{(4n)r},\\
\mathcal{P}^{-}(D^{2}u, Du)&\leq\epsilon_{1}/2~\text{in}~B^{+}_{(4n)r},\\
u_{x_n}&\leq g~\text{on}~B_{(4n)r}^0,\\
\|g\|_{L^{\infty}(B_{4n}^0)}&\leq \frac{\epsilon_{1}}{C_{1}\Lambda}\\
\text{and}~|\{u>r^{\alpha}\}\cap Q^{+}_{r}|&\geq\frac{1}{2}|Q^{+}_{r}|,
\end{aligned}
\right.
\end{equation}
we have,
\begin{equation}\label{lar8}
u> c_0r^{\alpha}  \\\\
\end{equation}
\text{in $Q^{+}_{3r}$.} In particular, $u >  c_0 r^\alpha$ in $B_r^+$.
\end{thm}
\begin{proof}
Let $\tau>1$ be  such that 

\begin{equation}\label{violate}
C\tau^{-\epsilon}<\frac{|Q^{+}_{1}|}{2}
\end{equation}

 where $C$ and $\epsilon>0$ are the constants from the $L^{\epsilon}$ estimate as in Theorem \ref{bounepsilon} above. Now, consider the following function
\[\tilde{u}~:~B^{+}_{4n}\longrightarrow\R,\]
defined by
\begin{equation}\label{lar1}
\tilde{u}(x)=\tau r^{-\alpha}u(rx)+\frac{\tau\epsilon_{1}}{\Lambda C_{1}} r^{1-\alpha}(4n-x_{n}).
\end{equation}
where $\ve_1$ will be chosen later. Then we have that $\tilde u$ satisfies


\begin{equation}\label{m321}
\left\{
\begin{aligned}{}
&\tilde{u}\geq0~~\text{in}~~B^{+}_{(4n)},\\
\mathcal{P}^{-}_{\lambda, \Lambda, \tilde \gamma}(D^{2}\tilde{u}, D\tilde{u})&\leq\big[\frac{\epsilon_{1}\tau}{C_{1}}+\frac{\epsilon_{1}\tau}{2}\big]r^{2-\alpha}~\text{in}~B^{+}_{(4n)},\\
\tilde{u}_{x_n}&\leq 0~\text{on}~B_{4n}^0,\\
\end{aligned}
\right.
\end{equation}
with $\tilde{\gamma}=\big(\gamma\tau +\frac{2\epsilon_{1}\tau}{\Lambda C_{1}}\big)r^{1-\alpha}.$ Furthermore, we have
\begin{equation}\label{lar6}
|\{\tilde{u}>\tau\}\cap Q^{+}_{1}|\geq\frac{1}{2}|Q^{+}_{1}| \geq C \tau^{-\ve}.
\end{equation}
Now let us choose $\ve_1= \tau^{-1}.$ Then we have that
$\tilde{\gamma}=\big(\gamma\tau+\frac{2}{\Lambda C_{1}}\big)r^{1-\alpha}.$ We now fix $\alpha \in (0, 1/2)$. Then by choosing $r_0$ small enough we can ensure that 
\begin{equation}
\mathcal{P}^{-}_{\lambda, \Lambda, \tilde \gamma}(D^{2}\tilde{u}, D\tilde{u}) \leq 1
\end{equation}
Moreover with $\ve_0$ as in Theorem \ref{bounepsilon}, we note that in view of our choice of $C_1$ in Remark \ref{cot}, if we  have 
\[\gamma \leq \tilde{\epsilon_{0}} \myeq\big(\epsilon_{0} \ve_1 -\frac{2\ve_1}{\Lambda C_{1}}\big),\]
 then we can ensure that $\tilde \gamma \leq \ve_0$.

 In such a case,  necessarily we must have 
\begin{equation}\label{lbd}
\tilde u > 1\ \text{in}\ Q^{+}_{3},
\end{equation}
 otherwise by applying the $L^{\ve}$ estimate in  Theorem \ref{bounepsilon}, we will obtain a contradiction to \eqref{lar6}.   We thus obtain from \eqref{lbd} that
\begin{equation}\label{la10}
u>\epsilon_{1}r^{\alpha}-C_{2}\epsilon_{1}r
\end{equation}
\text{in $Q^{+}_{3r}$. }The desired estimate \eqref{lar8} now follows from \eqref{la10} in a standard way provided $r_0$ is adjusted further depending also on $C_2$.
\end{proof}
With Theorem \ref{unif} in hand, we can now repeat the arguments in \cite{IS} to conclude the H\"older decay of $u$ at a boundary point.  The H\"older regularity upto the boundary consequently follows by a standard real analysis argument by combining the boundary estimate with the interior estimate  in \cite{IS}. We close this section by stating such a result. 

\begin{thm}\label{holder1}
For any continuous function $u:\overline{B^{+}_{1}} \longrightarrow\R,$ such that
\begin{equation*}
\left\{
\begin{aligned}{}
\mathcal{P}^{-}(D^{2}u, Du)&\leq C_{0}~~\text{in}~B^{+}_{1},\\
\mathcal{P}^{+}(D^{2}u, Du)&\geq -C_{0}~~\text{in}~B^{+}_{1},\\
u_{x_n}&=g~~\text{on}~~B_1^0,\\
\|g\|_{L^{\infty}(B_1^0)}&\leq C_{0},
\end{aligned}
\right.
\end{equation*}
we have $u\in C^{\alpha}(\overline{B^{+}_{\frac{1}{2}}})$ for some $\alpha>0$ depending on $\lambda, \Lambda$ and the dimension.
\end{thm}
\section{Equicontinuous estimates upto the boundary  for equations which are uniformly elliptic when the gradient is small}\label{lga}
In this section we obtain equicontinuous estimates for equations of the type \eqref{it} for large slopes, i.e. when $|p|$ is large.
As we have already mentioned in the introduction, since an appropriate generalization     of the doubling variable argument of Ishii and Lions  to our Neumann problem  is not clear to us, therefore we instead adapt the method of  Savin as in \cite{sa} based on sliding paraboloids.  

 \medskip

Now    in order to see that the method of sliding paraboloids  can be applied in this situation ( which is tailor-made for equations which are uniformly elliptic when the gradient is small),  we note that \eqref{it} can  be rewritten as
 \[
 \bigg|\frac{Du}{|p|} + \frac{p}{|p|} \bigg|^\beta F(D^2u) = \frac{f}{|p|^{\beta}}.
 \]
Therefore, for large enough $|p|$,  getting equicontinuity estimates for \eqref{it} reduces to  getting such estimates  for equations of the following type

\begin{equation}\label{sm1}
\left\{
\begin{aligned}{}
|e+ \sigma Du|^{\beta} F(D^{2}u,x)&= f~\text{in}~B^{+}_{1},\\
u_{x_n}&=g~\text{on}~T_1,\\
\end{aligned}
\right.
\end{equation}
where $|e|=1$,  $0< \sigma \leq 1$  and $ F:S(n) \times \R^{n} \longrightarrow\R,$  is a uniformly elliptic operator, i.e.
\begin{equation}\label{uni}
\lambda \|Y\|\leq F (X+Y,x)-F(X,x)\leq \Lambda \|Y\|,
\end{equation}
for all $X, Y\in S(n)$ with $Y\geq0.$  Note that the equation in \eqref{sm1} has a uniformly elliptic structure when $|Du|$ is small ( say when $|Du| \leq \frac{1}{2 \sigma}$).

\medskip

In our discussion, we will however be considering slightly more general degenerate elliptic operators as in \cite{sa}. 
More precisely, we  consider fully nonlinear operators of the type $ \tilde F: S(n) \times \R^n \times \R^{n} \longrightarrow\R,$ which satisfies the following structural conditions
\begin{enumerate}
\item[SC1)]~~ $\tilde F$ is degenerate elliptic, that is,
\[ \tilde F(X+Y,q,x)\geq  \tilde F(X,q,x)~~\text{for~all}~X,Y\in S(n),~Y\geq0~\text{and}~~(q,x)\in\R^{n}\times\R^{n}.\]
\item[SC2)]~~$\tilde F(0,q,x)=0$~~for all $(q,x)\in\R^{n}\times\R^{n}.$
\item[SC3)]~~$\tilde F$ is uniformly elliptic in a small neighbourhood of $0,$ that is, there is a $\delta>0$ such that
\[\lambda\|Y\|\leq \tilde F(X+Y,q,x)-\tilde F(X,q,x)\leq\Lambda\|Y\|,\]
for some $0< \lambda< \Lambda$,  $q\in B_{\delta}$, $X,Y\in S(n)$ and $Y\geq0$ and $x\in\R^{n}.$
\end{enumerate}
Note that  it is clear that the operator $\tilde F(X,q, x)=|e+\sigma q|^{\beta}F(X,x)$ satisfies the structural conditions SC1), SC2) and SC3) with  ellipticity bounds $(\frac{1}{2})^{\beta}\lambda$ and $(\frac{3}{2})^{\beta}\Lambda$ for $\delta= \frac{1}{2\sigma}$. 
\medskip

Let  us now  consider the following problem:
\begin{equation}\label{m120}
\left\{
\begin{aligned}{}
\tilde F(D^{2}u,Du,x)&\leq f~\text{in}~B^{+}_{1},\\
u_{x_n}&\leq0~\text{on}~T_1,\\
\end{aligned}
\right.
\end{equation}
where $\tilde F$ satisfies SC1)-SC3). The following lemma is a boundary version of Lemma 2.3 in \cite{CF} which in turn is inspired by the ideas in the proof of  Lemma 2.1 in \cite{sa}.
\begin{lem}\label{smesti}
Let $u$ be a viscosity solution to  \eqref{m120}. Fix $a\in(0,\delta/2),$ let $B\subset B^{+}_{1}$ be a compact set, and define $A\subset \overline{{B^{+}_{1}}}$ to be the set of contact points of paraboloid with vertices in $B$ and opening $-a$, namely  the set of points $x\in \overline{{B^{+}_{1}}}$ such that there exists $y\in B$ which satisfies \begin{equation}
\inf_{\xi\in B^{+}_{1}}\big\{\frac{a}{2}|y-\xi|^{2}+u(\xi)\big\}=\frac{a}{2}|y-x|^{2}+u(x).
\end{equation}
Then there exists universal constant $c_{1}>0$ such that
\begin{equation}
c_{1}|B|\leq|A|+\int_{A} \frac{|f(x)|^n}{a^{n}} dx
\end{equation}
\end{lem}

\begin{proof}
Since $B$ is compact subset of $B^{+}_{1},$ therefore for any $y\in B,$ $y_{n}>0.$ Therefore the contact point $x\not\in  B_1^0.$ For if $x\in A\cap B_1^0,$ then the paraboloid
\[P^{y}(\xi)=u(x)+\frac{a}{2}|x-y|^{2}-\frac{a}{2}|\xi-y|^{2}\]
touches $u$ at $x\in B_1^0$ from below and also $(P^{y})_{\xi_n}(x)=ay_{n}>0,$ which contradicts   the Neumann condition  in the viscosity formulation as in \eqref{m120} above. At this point, we can essentially repeat the arguments as in   Lemma 2.3 in  \cite{CF}. Note that although Lemma 2.3 in \cite{CF} deals with $C^{2}$ solutions, but nevertheless the proof  can be  generalized to semiconcave solutions using Alexandrov's theorem and then to arbitrary viscosity solutions using inf convolution. See for instance the proof of Lemma 2.1 in \cite{sa}. 

\end{proof}


Before stating our next result, we first introduce the following notation.
\begin{equation}
A_{a}=\Big\{x\in B^{+}_{1}~~\Big|~~u(x)\leq a,~\text{ and~there~exists}~y\in\overline{B^{+}_{1}}~\text{such~that}~\inf_{z\in B^{+}_{1}}\big[u(z)+\frac{a}{2}|y-z|^{2}\big]=u(x)+\frac{a}{2}|y-x|^{2}\Big\}.
\end{equation}
Namely, $A_a$ is the set of all $x \in B_1^{+}$ such that $u(x) \leq a$ and the function $u$ can be touched from below at $x$  with a  paraboloid of opening $-a$ with vertex in $\overline{B_1^+}$.
The next result is the boundary version of the Lemma 2.4 in \cite{CF} . See also the corresponding  Lemma 2.2  in \cite{sa}.
\begin{lem}\label{smb}
Let $u$ be as in \eqref{m120}. Also let $a>0$ and $x_{0}\in B_1^0$ such that $B^{+}_{4r}(x_{0})\subset B^{+}_{1}.$ Then there exist universal constants $C_{b}$ and $c_{b}$ and $\mu_{b},$ such that if $a\leq\frac{\delta}{C_{b}},$ $\|f\|_{L^{\infty}(B^{+}_{1})}\leq \mu_{b}a$ and
\begin{equation}\label{sm6'}
B^{+}_{r}(x_{0})\cap A_{a}\not=\emptyset,
\end{equation}
then
\begin{equation}\label{sm32}
c_{b}|B_{r}^+(x_{0})|\leq|B_{\frac{r}{16}}^+(x_{0}) \cap A_{a C_{b}}|.
\end{equation}
\end{lem}
\begin{proof}
By \eqref{sm6'}, there exists $x_{1}\in B^{+}_{r}(x_{0})\cap A_{a}$.  So by  the definition of $A_{a}$,  there  exists  $y_{1}\in B^{+}_{1}$ such that the paraboloid
\begin{equation}
Q_{y_{1}}(\xi)=u(x_{1})+\frac{a}{2}|x_{1}-y_{1}|^{2}-\frac{a}{2}|\xi-y_{1}|^{2},
\end{equation}
satisfies
\begin{equation}\label{m12}
\left\{
\begin{aligned}{}
Q_{y_{1}}(\xi)&\leq u(\xi)~~\forall~\xi\in B^{+}_{1},\\
Q_{y_{1}}(x_{1})&=u(x_{1}).
\end{aligned}
\right.
\end{equation}
We now make the following claim.

\medskip

\textbf{Claim:} There exists $z\in B^{+}_{\frac{r}{32}}(x_{0})$ such that
\begin{equation}\label{sm13}
u(z)\leq Q_{y_{1}}(z)+C_{1}ar^{2}.
\end{equation}
In order to prove the claim,  let us consider the function $\phi:\R^{n}\longrightarrow\R,$ defined by
\begin{equation}\label{smphi}
    \phi(x)=
\begin{cases}
    \frac{1}{\alpha}(32^{\alpha}-1), & \text{if}~|x|<\frac{1}{32}\\
    \frac{1}{\alpha}(|x|^{-\alpha}-1),& \text{if }~\frac{1}{32}\leq|x|\leq1\\
    0, &\text{if}~|x|>1
\end{cases}
\end{equation}
where $\alpha$ is  to be chosen later. In terms of $\phi$, we then define $\psi:B^{+}_{r}(x_{0})\longrightarrow\R$  in the following way, 
\begin{equation}
\psi(x)=Q_{y_{1}}(x)+ar^{2}\phi\big(\frac{x-x_{0}}{r}\big)-\epsilon ar^{2}(r-x_{n}),
\end{equation}
where $\epsilon$ is a sufficiently small number which will be  chosen below. We note that for $x$ satisfying $\frac{r}{32}<|x-x_{0}|<r$, the function $\psi$  is smooth. Moreover, for any $x$ in the above set we have
\begin{equation}
D\psi(x)=-a(x-y_{1})+arD\phi\big(\frac{x-x_{0}}{r}\big)+\epsilon ar^{2}e_n.
\end{equation}
Thus it follows that
\begin{equation}\label{m12}
\begin{aligned}{}
|D\psi(x)|&\leq 4a+a\frac{r^{1+\alpha}}{|x-x_{0}|^{\alpha}}\\
&\leq a(4+32^{1+\alpha})<\delta,
\end{aligned}
\end{equation}
provided $C_{b}\geq(4+32^{1+\alpha})$ and  consequently $F$ is uniformly elliptic in the above region. In view of $SC3)$ we have
\begin{equation}\label{sm9}
\begin{aligned}{}
\tilde F(D^2\psi(x),D\psi(x),x)-f(x)&\geq\lambda||(D^{2}\psi(x))^{+}||-\Lambda\|(D^{2}\psi(x))^{-}\|-\|f\|_{L^{\infty}(B^{+}_{1})}\\
&\geq a\Big[\big(\lambda(1+\alpha)-\sqrt{n-1}\Lambda\big)\frac{r^{\alpha+2}}{|x-x_{0}|^{\alpha}}-\lambda-\sqrt{n-1}\Lambda-\mu_{b}\Big].
\end{aligned}
\end{equation}
Consequently,  if we choose $\alpha$ sufficiently large, then we obtain
\begin{equation}
\hspace{2cm}~\tilde F(D^2\psi(x),D\psi(x),x)-f(x)>0~~~\text{in}~~B^{+}_{r}(x_{0})\cap\{x_{n}>0\}\cap\{\frac{r}{32}<|x-x_{0}|<r\}.
\end{equation}
Also for  $\bar{x}\in B_1^0$, we observe that 
\begin{equation}\label{sm10}
\partial_{x_n}\psi(\bar{x})=a(y_{1})_{n}+a\epsilon r^{2}>0.
\end{equation}
We denote by $z$ the point  where $\min_{x\in\overline{B^{+}_{r}}(x_{0})}(u-\psi)$ is achieved. We now   choose $\epsilon>0$ sufficiently small such that
\begin{equation}\label{smol}
-ar^{2}\phi\big(\frac{x_{1}-x_{0}}{r}\big)+ar^{2}(r-(x_{1})_{n})\epsilon<0.
\end{equation}
Note that although the choice of $\ve$ depends on $x_1$ but as we will see, it doesn't affect   the final conclusion. 
\eqref{smol} implies
\begin{equation}\label{sm11}
u(x_{1})-\psi(x_{1})=Q_{y_{1}}(x_{1})-\psi(x_{1})=-ar^{2}\phi\big(\frac{x_{1}-x_{0}}{r}\big)+ar^{2}(r-(x_{1})_{n})\epsilon<0.
\end{equation}
Moreover  on $\partial B_{r}(x_{0})\cap\{x_{n}>0\},$  we have
\begin{equation*}\label{sm12}
\begin{aligned}{}
u(x)&\geq Q_{y_{1}}(x)\\
&\geq Q_{y_{1}}(x)-\epsilon ar^{2}(r-x_{n})~~~~~~\hspace{0.5cm}\big(\text{since}~~(ar^{2}(r-x_{n})\geq0)\big)\\
&=\psi(x).
\end{aligned}
\end{equation*}
Now we note that since  $u_{x_n}\leq0$~on $B_1^0$~(in the viscosity sense), so in view of \eqref{sm10},  we can deduce that $u-\psi$ cannot attain minimum on $\big\{ \frac{r}{32}<|x-x_{0}|<r \big\}\cap\big\{x_{n}=0\big\} \cup \partial B_r (x_0) \cap \{x_n >0\}.$ Therefore there exists $z\in B^{+}_{\frac{r}{32}}(x_{0})$ such that
\begin{equation*}\label{sm12}
\begin{aligned}{}
u(z)&< \psi(z),\hspace{0.5cm}~~~~\big(\text{thanks to}~\eqref{sm11}\big)\\
&\leq Q_{y_{1}}(z)+ar^{2}\phi\big(\frac{z-x_{0}}{r}\big)-\epsilon ar^{2}(r-z_{n})\\
&\leq Q_{y_{1}}(z)+ar^{2}\phi\big(\frac{z-x_{0}}{r}\big)~\hspace{0.5cm}~~\big(\text{since}~~(ar^{2}(r-z_{n})\geq0)\big)\\
&\leq Q_{y_{1}}(z)+C_{1}ar^{2}.
\end{aligned}
\end{equation*}
For a given $L >0$ and  $y\in B_{\frac{r}{128}}(z)\cap\{y_{n}>z_{n} \}$, we  consider the paraboloid
\begin{equation}\label{py}
P_{y}(x)=Q_{y_{1}}(x)-L\frac{a}{2}|x-y|^{2}.
\end{equation}
It is easy to check that for each $y$, $P_{y}$ is a paraboloid with opening $-(L+1)a$ and vertex $\frac{y_{1}+Ly}{1+L}.$ We slide it from below till it touches the graph of $u$ for the first time. We claim that the contact point $\bar{x}\in B^+_{\frac{r}{32}}(z)$ provided $L$ is large enough. In order to prove such a claim, we make the following observations.

\medskip

(i)~ Suppose $\bar{x}\in B_1^0,$ then
\begin{equation}
\partial_{x_n}(P_{y})(\bar{x})=a(y_{1})_{n}+La y_{n}>La(z_{n})>0.
\end{equation}
Now since $\partial_{x_n}u\leq0$ on $B_1^0$ (in the viscosity sense), therefore  $P_{y}$ cannot touch $u$ from below at points in $ B_1^0.$\\
(ii)~Suppose instead   $\bar{x}$ satisfies  $|\bar{x}-z|\geq\frac{r}{32},$ then   using
 $u\geq Q_{y_{1}}$ on $B^{+}_{1},$ we find that the following holds,
\begin{equation}\label{sm14}
u(\bar{x})-Q_{y_{1}}(\bar{x})+\frac{La}{2}|\bar{x}-y|^{2}\geq \frac{La}{2}\big(\frac{r}{32}\big)^{2}.
\end{equation}
On the other hand since 
\begin{equation}\label{sm15}
\begin{aligned}{}
&\min_{\overline{B^{+}_{1}}}\big\{u(x)-Q_{y_{1}}(x)+\frac{La}{2}|x-y|^{2}\big\}\\
&\leq u(z)-P_{y_{1}}(z)+\frac{La}{2}|y-z|^{2}\\
&\leq C_{1}ar^{2}+\frac{La}{2}\big(\frac{r}{128}\big)^{2},
\end{aligned}
\end{equation}
thus by choosing $L$ large enough and by taking  into account\eqref{sm14} and \eqref{sm15}, we find that the contact point $\bar{x}\in B_{\frac{r}{32}}^+(z)\subset B^{+}_{\frac{r}{16}}(x_{0}).$\\
We now show that at the contact point $\bar{x},$  we have $u(\bar{x})\leq La$ provided $L$ is further adjusted.  Indeed, since  \[Q_{y_{1}}(x_{1})=u(x_{1})\leq a\] and also
\begin{equation*}\label{sm16}
\begin{aligned}{}
Q_{y_{1}}(\bar{x})&=u(x_{1})+\frac{a}{2}|x_{1}-y_{1}|^{2}-\frac{a}{2}|\bar{x}-y_{1}|^{2}\\
&\leq a+4a=5a,
\end{aligned}
\end{equation*}
hence from \eqref{sm15} (since $\bar{x}$ is the point where the minimum in \eqref{sm15} is achieved), we find
\begin{equation*}\label{sm17}
\begin{aligned}{}
u(\bar{x})&\leq Q_{y_{1}}(\bar{x})-L\frac{a}{2}|\bar{x}-y|^{2}+C_{1}ar^{2}+\frac{La}{2}(\frac{r}{128})^{2}\\
&\leq 5a+Car^{2}+\frac{La}{2}(\frac{r}{128})^{2}\leq La,
\end{aligned}
\end{equation*}
provided $L$ is sufficiently large. Now  as $y$ varies in $B_{\frac{r}{128}}(z)\cap\{y_{n}\geq z_{n}\},$ the set of vertices of the paraboloids as in \eqref{py} falls in the  region \begin{equation}
\tilde{R}\stackrel{\text{def}}{=}\Big[B\Big(\frac{y_{1}+Lz}{1+L},~~\frac{Lr}{128(1+L)}\Big)\bigcap\Big\{\xi_{n}\geq\frac{(y_{1})_{n}+Lz_{n}}{1+L}\Big\}\Big],
\end{equation}
therefore by applying Lemma \ref{smesti}, we get
\begin{equation}\label{sm30}
\begin{aligned}{}
c_{0}|\tilde{R}|&\leq|B^{+}_{\frac{r}{16}(x_{0})}\cap A_{a(L+1)}|+\Big(\frac{\|f\|^{n}_{L^{\infty}(B^{+}_{1})}}{a^{n}}\Big)|B^{+}_{\frac{r}{16}(x_{0})}|\\
&\leq |B^{+}_{\frac{r}{16}(x_{0})}\cap A_{a(L+1)}|+\big(\frac{\mu_{b}}{16}\big)^{n}|B^{+}_{r}(x_{0})|.
\end{aligned}
\end{equation}
Then we observe that 
\begin{equation}\label{sm31}
|\tilde{R}|=C|B^{+}_{r}(x_{0})|
\end{equation}
for some constant $C$ independent of $r.$  From \eqref{sm30} and \eqref{sm31} we finally  obtain
\[\Big[c_{0}C-\big(\frac{\mu_{b}}{16}\big)^{n}\Big]|B^{+}_{r}(x_{0})|\leq |B^{+}_{\frac{r}{16}(x_{0})}\cap A_{aL}|.\]
and thus the  conclusion of the lemma follows.
\end{proof}
We note that the  interior analogue of the lemma above is  crucially needed   to  apply the measure decay estimate in \cite{CF} and \cite{sa}  which is the key ingredient  needed to obtain  quantitative oscillation decay estimates.  In our situation,  in order to combine  the boundary and interior  estimate,  we also  need the  following  additional lemma.
\begin{lem}\label{smib}
Let $a>0,$ and suppose that $B_{4r}^+(x_{0})\subset B^{+}_{1}$ and $(x_{0})_{n}\geq\frac{r}{16}.$ Suppose that $u$ is a viscosity solution of \eqref{m120}. Then there exists  universal constants $C_{ib},$ $c_{ib}$ and $\mu_{ib}>0$ such that if
\begin{equation*}\label{sm18}
\left\{
\begin{aligned}{}
&\|f\|_{L^{\infty}(B^{+}_{1})}\leq a\mu_{ib},\\
&a\leq\frac{\delta}{C_{ib}}~~\text{and}\\
&B_{r}^+(x_{0})\cap A_{a}\not=\emptyset,
\end{aligned}
\right.
\end{equation*}
then
\begin{equation}
|B_{\frac{r}{16}}^+(x_{0})\cap A_{a C_{ib}}|\geq c_{ib}|B_{r}^+(x_{0})|.
\end{equation}
 \end{lem}
 
 \begin{proof}
The proof of this Lemma  is similar to that of  Lemma \ref{smb}. We nevertheless  give a sketch of it  for the sake of completeness.
 
By our assumption, there exists $x_{1}\in(B_{r}(x_{0})\cap \R^{n}_{+})\cap A_{a}.$ So from the definition of $A_{a},$ for some  $y_{1}\in B^{+}_{1}$, we have that the paraboloid
\begin{equation}
Q_{y_{1}}(\xi)=u(x_{1})+\frac{a}{2}|y_{1}-x_{1}|^{2}-\frac{a}{2}|y_{1}-\xi|^{2},
\end{equation}
satisfies
\begin{equation*}\label{sm19}
\left\{
\begin{aligned}{}
u(\xi)&\geq Q_{y_{1}}(\xi)~~\forall~~~\xi\in B^{+}_{1},\\
u(x_{1})&=Q_{y_{1}}(x_{1}).
\end{aligned}
\right.
\end{equation*}
We now claim that there exists  $z\in B_{\frac{r}{16}}(x_{0})\subset B^{+}_{1}$ (since $(x_{0})_{n}\geq\frac{r}{16}$) such that
\begin{equation}
u(z)\leq Q_{y_{1}}+C_{2}ar^{2}.
\end{equation}
for some universal $C_2$. In order to prove the claim, we consider the following  function $\Psi:\overline{B_{r}(x_{0})}\longrightarrow\R,$ defined by
\begin{equation}
\Psi(x)=Q_{y_{1}}(x)+ar^{2}\phi\big(\frac{x-x_{0}}{r}\big),
\end{equation}
with  $\phi$ as in \eqref{smphi}. Again we can choose  $\alpha$ large enough so that  the following  differential inequality is ensured
\begin{equation}\label{smo}
\left\{
\begin{aligned}{}
F(D^2\Psi,D\Psi,x)&> f(x)~~\text{in}~~\{x~|~\frac{r}{32}<|x-x_{0}|<r\}\cap\{x_{n}>0\},\\
\Psi_{x_n}&>0~~\text{on}~~\{x~|~\frac{r}{32}<|x-x_{0}|<r\}\cap\{x_{n}=0\}.
\end{aligned}
\right.
\end{equation}
We only check the second condition since the first one is as in the previous lemma.  Suppose that $\frac{r}{32}<|\bar{x}-x_{0}|<r$ and  also that $\bar{x}_{n}=0.$ Then we have that
\begin{equation}\label{sm21}
\partial_{x_n}\Psi(\bar{x})=a(y_{1})_{n}+\Big(\frac{r}{|\bar{x}-x_{0}|}\Big)^{\alpha+2}\frac{(x_{0})_{n}}{r}>0.
\end{equation}
At this point, by arguing as in the proof of the previous lemma, we  conclude that the point of minimum in

\begin{equation}\label{sm23}
\min_{\overline{B_{r}(x_{0})}\cap\{x~|~x_{n}\geq0\}}\{u-\Psi\}
\end{equation}
is realized in $\overline{B_{r/32}(x_0)}$. The rest of the arguments can then be repeated and the conclusion of the lemma follows similarly. \end{proof}
Finally, we state the interior version of the above measure estimate. ( see Lemma 2.4 in \cite{CF}).
\begin{lem}\label{smi}
Let $u$ be a solution to the second order differential inequality  in \eqref{m120}. Let $a>0,$ and $B_{4r}(x_{0})\subset B^{+}_{1}.$ Then there exist universal constants $C_{i}$ and $c_{i}$ and $\mu_{i},$ such that if $a\leq\frac{\delta}{C_{i}},$ $\|f\|_{L^{\infty}(B^{+}_{1})}\leq \mu_{i}a$ and
\begin{equation}\label{sm6}
B_{r}(x_{0})\cap A_{a}\not=\emptyset,
\end{equation}
then
\begin{equation}\label{sm32}
c_{i}|B_{r}(x_{0})|\leq|B_{\frac{r}{16}}(x_{0})\cap A_{C_{i}a}|.
\end{equation}
\end{lem}

\subsection*{Boundary version of measure~decay}
We now prove a  boundary version of the  covering  lemma  that corresponds to  lemma 2.3  in \cite{sa}. Similar to the interior case, such a covering lemma is one of the crucial ingredients in our proof of the oscillation decay estimate as in Theorem \ref{rescaled} below.
\begin{lem}\label{smdecay}

Let $D_{0},D_{1}$ be two closed sets satisfying
\[\emptyset\not=D_{0}\subset D_{1}\subset\overline{B^{+}_{r_{0}}}.\]
and $\sigma_{1},\sigma_{2},\sigma_{3} \in (0,1)$  be such that for $r_0 \leq \frac{1}{14}$,  the  following  hypotheses are satisfied,


\begin{equation*}
\text{H(I)}\left\{
\begin{aligned}{}
&\text{Whenever}~x\in B_{r_{0}}^0~\text{and for some}~r>0,~\text{one has}\\
&(i)~~B^{+}_{4r}(x)\subset B^{+}_{1},\\
&(ii)~~B^{+}_{\frac{r}{16}}(x)\subset B^{+}_{r_{0}},\\
&(iii)~~\overline{B^{+}_{r}(x)}\cap D_{0}\not=\emptyset,\\
&\text{then,}\\
&|B^{+}_{\frac{r}{16}}(x)\cap D_{1}|\geq\sigma_{1}|B^{+}_{r}(x)|.
\end{aligned}
\right.
\end{equation*}
\begin{equation*}
\text{H(II)}\left\{
\begin{aligned}{}
&\text{Whenever}~x\in B^{+}_{r_{0}}~\text{and for some}~r>0,~\text{one has}\\
&(i)~~x_{n}\geq\frac{r}{16},\\
&(ii)~~B_{4r}^+(x)\subset B^{+}_{1},\\
&(iii)~~B_{\frac{r}{16}}(x)\subset B^{+}_{r_{0}},\\
&(iv)~~\overline{B_{r}^+(x)}\cap D_{0}\not=\emptyset,\\
&\text{then,}\\
&|(B_{\frac{r}{16}}(x)\cap D_{1}|=|B_{\frac{r}{16}}^+(x)\cap D_{1}|\geq\sigma_{2}|B_{r}^+(x)|.
\end{aligned}
\right.
\end{equation*}
\begin{equation*}
\text{H(III)}\left\{
\begin{aligned}{}
&\text{Whenever}~x\in B^{+}_{r_{0}}~\text{and for some}~r>0,~\text{one has}\\
&(ii)~~B_{4r}(x)\subset B^{+}_{1},\\
&(iii)~~B_{\frac{r}{16}}(x)\subset B^{+}_{r_{0}},\\
&(iv)~~\overline{B_{r}(x)}\cap D_{0}\not=\emptyset,\\
&\text{then,}\\
&|(B_{\frac{r}{16}}(x)\cap D_{1}|\geq\sigma_{3}|B_{r}(x)|.
\end{aligned}
\right.
\end{equation*}
In that case, we  have that  the following estimate holds,
\begin{equation}
|B^{+}_{r_{0}}\setminus D_{1}|\leq(1-\sigma)|B^{+}_{r_{0}}\setminus D_{0}|,
\end{equation}
for some $\sigma\in (0,1).$
\end{lem}
\begin{proof}
Given $x_{0}\in B^{+}_{r_{0}}\setminus D_{0},$ set $\bar{r}=\text{dist}\{x_{0},D_{0}\}\leq 2r_{0}.$ Let us also define $r=\frac{8}{7}\bar{r}.$ We will first show that  for some $\sigma >0$, the following estimate holds,
\begin{equation}\label{cl1}
\begin{aligned}{}
|B_{\frac{r}{4}}(x_{0})\cap B^{+}_{r_{0}}\cap D_{1}|  \geq\sigma |B_{r}(x_{0})\cap B^{+}_{r_{0}}(x_{0})|.
\end{aligned}
\end{equation}

The proof of \eqref{cl1} is based on a case by case argument  depending on the distance  of $x_{0}$ from $\{x_n=0\}$.  Note that there are $4$ possibilities.
\noindent \begin{enumerate}
 \item[Case (i)]~$x_{0}\in B_{r_{0}}^0$.
 \item[Case (ii)]~$0<(x_{0})_{n}<\frac{r}{16}=\frac{\bar{r}}{14}.$
 \item[Case (iii)]~$\frac{r}{16}\leq (x_{0})_{n}<r_{0}-\frac{r}{16}.$
 \item[Case (iv)]~$r_{0}-\frac{r}{16}\leq (x_{0})_{n}\leq r_{0}.$
\end{enumerate}
\text{Case-(i)} In this case let us define
\[x_{1}=x_{0}-\frac{r}{16}\frac{x_{0}}{|x_{0}|}\in T_{r_{0}}.\]
when $|x_0| >0$. Otherwise, we take $x_1=x_0$. Then it is easy to observe that the following hold:\\
(a)~$B^{+}_{\frac{r}{16}}(x_{1})\subset B^{+}_{r_{0}},$\\
(b)~~$B^{+}_{\frac{r}{16}}(x_{1})\subset B^{+}_{\frac{r}{8}}(x_{0}).$\\
(c)~$B^{+}_{r}(x_{1})\cap D_{0}\not=\emptyset.$  

\medskip

(a) and (b) are easy consequences of triangle inequality.  (c) can be seen as follows. Since $\bar{r}=\text{dist}\{x_{0},D_{0}\}$, therefore there exists   $z_{0}\in D_{0}$ such that $|x_{0}-z_{0}|=\bar{r}.$ Thus
\begin{equation*}
\begin{aligned}{}
|z_{0}-x_{1}|&\leq|z_{0}-x_{0}|+|x_{0}-x_{1}|\\
&< \bar{r}+\frac{\bar{r}}{14}=\frac{15\bar{r}}{14}<\frac{16\bar{r}}{14}=r.
\end{aligned}
\end{equation*}
This implies that $z_{0}\in B^{+}_{r}(x_{1})$ and hence $z_{0}\in B^{+}_{r}(x_{1})\cap D_{0}.$
Then we observe that the following holds,

\medskip
(d)~~$B^{+}_{4r}(x_{1})\subset B^{+}_{1}.$

\medskip

 In fact, since $(x_{1})_{n}=0,$ $|x_{1}|\leq r_{0},$ and $r\leq 3r_{0},$ therefore if $x\in B^{+}_{4r}(x_{1}),$  then
\[|x|\leq|x-x_{1}|+|x_{1}|<4r+r_{0}\leq 13r_{0} < 1.\]
Therefore in this situation we  see that the conditions in  H(I) are satisfied and consequently we have 
\begin{equation}\label{sm33}
|B^{+}_{\frac{r}{16}}(x_{1})\cap D_{1}|\geq\sigma_{1}|B^{+}_{r}(x_{1})|.
\end{equation}
Thus from \eqref{sm33}, we find
\begin{equation}\label{sm34}
\begin{aligned}{}
\sigma_{1}|B_{r}(x_{0})\cap B^{+}_{r_{0}}|&\leq\sigma_{1}|B^{+}_{r}(x_{0})|\\
&=\sigma_{1}|B^{+}_{r}(x_{1})|~~(\text{since the measure is translation invariant})\\
&\leq|B^{+}_{\frac{r}{16}}(x_{1})\cap D_{1}|~~(\text{by}~~\eqref{sm33})\\
&\leq |B^{+}_{\frac{r}{8}}(x_{0})\cap B^{+}_{r_{0}}\cap D_{1}|~~(\text{by~observation~}~(a),(b))\\
&\leq|B_{\frac{r}{8}}(x_{0})\cap B^{+}_{r_{0}}\cap D_{1}|\\
&\leq|B_{\frac{r}{4}}(x_{0})\cap B^{+}_{r_{0}}\cap D_{1}|.
\end{aligned}
\end{equation}
\eqref{cl1} thus follows in this case. We now consider Case (ii).\\
In this case we have $0<(x_{0})_{n}<\frac{\bar{r}}{14}=\frac{r}{16}.$ Let us consider the following shifted point corresponding to $x_0$.
\begin{equation}\label{smphi}
x_{1}=
\begin{cases}
    P(x_{0})-\frac{\bar{r}}{14}\frac{P(x_{0})}{|P(x_{0})|}, & \text{if}~P(x_{0})\not=0\\
    0,& \text{if }~P(x_{0})=0
\end{cases}
\end{equation}
\\

where $P(x_{0})$ is the projection of $x_{0}$ on $\{x\in\R^{n}~~|~~x_{n}=0\}.$ We first note  that $(x_{1})_{n}=0.$  
Moreover we  easily observe that the following hold,
\begin{enumerate}
\item[(a')]~$B^{+}_{\frac{r}{16}}(x_{1})\subset B^{+}_{r_{0}}.$
\item[(b')]~$\overline{B_{r}^{+}(x_{1})}\cap D_{0}\not=\emptyset.$ 
\item[(c')]~$B^{+}_{\frac{r}{16}}(x_{1})\subset B_{\frac{r}{4}}^+(x_{0}) \subset B_{\frac{r}{4}}(x_{0})$.
\item[(d')]~$B^{+}_{4r}(x_{1})\subset B^{+}_{14r_{0}}\subset B^{+}_{1}$ since $r_{0}\leq\frac{1}{14}.$ \end{enumerate}
(a'), (c') and (d') follow easily from triangle inequality. (b') can be seen as follows. As in Case i),  let $z_{0}\in D_{0}$ be such that $|x_{0}-z_{0}|=\bar{r}.$ Then 
\begin{equation*}
\begin{aligned}{}
|x_{1}-z_{0}|&\leq|x_{1}-P(x_{0})|+|P(x_{0})-x_{0}|+|x_{0}-z_{0}|
&<\frac{\bar{r}}{14}+\frac{\bar{r}}{14}+\bar{r}=\frac{8\bar{r}}{7}=r,
\end{aligned}
\end{equation*}
(b') thus follows.

In view of the observations (a'),(b') and (d') and  (HI), we get
\begin{equation}\label{sm35}
|B^{+}_{\frac{r}{16}}(x_{1})\cap D_{1}|\geq \sigma_{1}|B^{+}_{r}(x_{1})|.
\end{equation}
We then note that
\begin{enumerate}
\item[(a")]~$|B^{+}_{r}(x_{1})|=|B^{+}_{r}(P(x_{0}))|$~\hspace{0.5cm}~\big(because $(x_{1})_{n}=(P(x_{0}))_{n}=0$\big).
\item[(b")]~$|B^{+}_{r}(P(x_{0}))|=|B^{+}_{r}(x_{0})\cap\{x~|~x_{n}\geq (x_{0})_{n}\}|$~\hspace{0.5cm}\big( because the  measure is translation invariant\big).
\item[(c")]~$|B^{+}_{r}(x_{0})\cap\{x~|~x_{n}\geq (x_{0})_{n}\}|=\frac{1}{2}|B_{r}(x_{0})|.$
\end{enumerate}
Thus
\begin{equation}\label{sm36}
\begin{aligned}{}
|B_{\frac{r}{4}}(x_{0})\cap B^{+}_{r_{0}}\cap D_{1}|&\geq |B^{+}_{\frac{r}{16}}(x_{1})\cap D_{1}|~\hspace{0.5cm}\big(\text{by}~(c')\big)\\
&\geq\sigma_{1}|B^{+}_{r}(x_{1})|~~\hspace{0.5cm}\big(\text{by}~~\eqref{sm35}\big)\\
&=\sigma_{1}|B^{+}_{r}(P(x_{0}))|~~\hspace{0.5cm}\big(\text{by}~~(a")\big)\\
&=\frac{\sigma_{1}}{2}|B_{r}(x_{0})|~~\hspace{0.5cm}\big(\text{by}~(b")~\text{and}~(c")\big)\\
&\geq\frac{\sigma_{1}}{2}|B_{r}(x_{0})\cap B^{+}_{r_{0}}(0)|.
\end{aligned}
\end{equation}
\eqref{cl1} thus follows  in this case as well. 

\medskip

We now look  at Case (iii).
In this case similar to that of Case (ii),  we consider the following shifted point corresponding to $x_0$, 
\begin{equation*}
x_{1}=
\begin{cases}
    x_{0}-\frac{\bar{r}}{14}\frac{P(x_{0})}{|P(x_{0})|}, & \text{if}~P(x_{0})\not=0\\
    x_0,& \text{if }~P(x_{0})=0.
\end{cases}
\end{equation*}
We then make the following observations.
\begin{enumerate}
\item[(e')]~From the choice of $x_{1}$ and the fact $\frac{r}{16}<(x_{0})_{n}=(x_{1})_{n}<r_{0}-\frac{r}{16},$ we find that
\[B_{\frac{r}{16}}(x_{1})\subset B^{+}_{r_{0}}~~~\text{and}~~~B_{\frac{r}{16}}(x_{1})\subset B_{\frac{r}{4}}(x_{0}).\]
\item[(f')]~By arguing as in the previous case, we also have 
\[\overline{B_{r}^+(x_{1})}\cap D_{0}\not=\emptyset.\]
\item[(h')]Likewise  we have $B_{4r}^+(x_{1})\subset B^{+}_{14r_{0}}\subset B^{+}_{1}(0).$
\end{enumerate}
So in view of above observations (e'), (f') and (h') , we find that the conditions in H(II) are satisfied and consequently we have
\begin{equation}\label{sm37}
|B_{\frac{r}{16}}(x_{1})\cap D_{1}|\geq\sigma_{2}|B_{r}^+(x_{1})|.
\end{equation}
Now in  order to get appropriate measure estimate in terms of ball centered at $x_{0}$ instead of $x_{1},$ let us also  observe that
\begin{enumerate}
\item[(d")]~Since $(x_{0})_{n}=(x_{1})_{n},$ hence
\[|B_{r}^+(x_{1})|=|B_{r}^+(x_{0})|.\]
\end{enumerate}
Therefore, we have
\begin{equation}\label{sm39}
\begin{aligned}{}
|\big(B_{\frac{r}{4}}(x_{0})\cap B^{+}_{r_{0}}\big)\cap D_{1}|&\geq |B_{\frac{r}{16}}(x_{1})\cap D_{1}|~\hspace{0.5cm}~\big(\text{by}~(e')\big)\\
&\geq\sigma_{2}|B_{r}^+(x_{1})|~\hspace{0.5cm}~\big(\text{by}~~\eqref{sm37}\big)\\
&=\sigma_{2}|B_{r}^+(x_{0})|~~\hspace{0.5cm}\big(\text{by}~~(d")\big)\\
&\geq\sigma_{2}|B_{r}(x_{0})\cap B^{+}_{r_{0}}|~\hspace{0.5cm}~
\end{aligned}
\end{equation}
We  finally note that  Case (iv) corresponds to the interior case and therefore by repeating the arguments  as in \cite{CF} ( given that H(III) holds) we will have
\begin{equation}\label{sm41}
\begin{aligned}{}
|B_{\frac{r}{4}}(x_{0})\cap B^{+}_{r_{0}}\cap D_{1}| \geq \sigma_{3}|B_{r}(x_{0})\cap B^{+}_{r_{0}}(x_{0})|.
\end{aligned}
\end{equation}
Thus in view of \eqref{sm34}, \eqref{sm36} \eqref{sm39} and \eqref{sm41}, it is clear that the estimate in \eqref{cl1} follows by letting  $\sigma=\min\{\sigma_{1}/2,~\sigma_{2},\sigma_{3}\}.$

\medskip

Now, for every $x\in B^{+}_{r_{0}}\setminus D_{0},$ we consider the ball centered at $x$ of  radius $r:=\text{dist}\{x,~D_{0}\}$. Then  by applying  Vitali covering's Lemma to this family,  we can  extract a subfamily $\{B_{r_{j}}(x_{j})\}$ such that the balls $\{B_{\frac{r_{j}}{3}}(x_{j})\}$ are disjoint. In particular, $\{B_{\frac{r_{j}}{4}}(x_{j})\}'s$ are disjoint. Hence,
\begin{equation}\label{sm42}
\begin{aligned}{}
|B^{+}_{r_{0}}\setminus D_{0}|&\leq  \sum_{j}|\Big(B_{r_{j}}(x_{j})\cap B^{+}_{r_{0}}\Big)\setminus D_{0}|\\
&\leq\sigma^{-1}\sum_{j}|\Big(B_{\frac{r_{j}}{4}}(x_{j})\cap B^{+}_{r_{0}}\Big)\cap(D_{1}\setminus D_{0})|\\
&\leq \sigma^{-1} |B^{+}_{r_{0}}\cap(D_{1}\setminus D_{0})|.
\end{aligned}
\end{equation}
From  \eqref{sm42} it follows that, 
\begin{equation}\label{sm43}
\begin{aligned}{}
|B^{+}_{r_{0}}\setminus D_{1}|&=|B^{+}_{r_{0}} \setminus D_0|-|B^{+}_{r_{0}}\cap (D_{1}  \setminus D_0)|\\
& \leq (1-\sigma) |B^{+}_{r_0} \setminus D_0|.
\end{aligned}
\end{equation}
This finishes the proof.
\end{proof}
Now, we are ready to prove the main oscillation decay result in this section. Before stating such a result, we make the following remarks.
\begin{rem}\label{smr1}
From now on, we let  $C=\max\{C_{i}, C_{ib},C_{b}\},$ $c=\min\{c_{i},c_{ib},c_{b}\}$ and $\mu=\min\{\mu_{i},\mu_{ib},\mu_{b}\},$ where triplet $(C_{b},c_{b},\mu_{b}), (C_{ib},c_{ib},\mu_{ib})$ and $(C_{i},c_{i},\mu_{i})$ are respectively from  the lemmas \ref{smb}, \ref{smib} and \ref{smi}. It is clear from the proofs  that if we replace such triplets in the hypothesis of the  respective Lemmas by $(C, c,\mu)$ then we get that the concluding inequality holds in all lemmas with $A_{Ca}$ instead of  $A_{C_{b}a},$ $A_{C_{ib}a}$ and $A_{C_{i}a}$.
\end{rem}

\begin{rem}
We would also like to remark that from here onwards, we would deal with the following non-homogeneous Neumann boundary value problem, 
\begin{equation}\label{sm44}
\left\{
\begin{aligned}{}
\tilde F(D^2u,Du,x)&= f~\text{in}~B^{+}_{1},\\
u_{x_n}&=g~\text{on}~B_1^0.\\
\end{aligned}
\right.
\end{equation}
\end{rem}
\begin{thm}\label{th1}
Let  $u\in C(B^{+}_{1}\cup B_1^0)$  be a viscosity solution \eqref{sm44} where $F$ satisfies the structure conditions SC1)-SC3) and $f \in C(\overline{B_1^+})$. Let $\lambda, \Lambda$ and $\delta$ be as in SC1)- SC3). Then there exist universal constants $\nu,\epsilon,\rho,\theta\in(0,1)$ such that if for some $\delta'$  satisfying $\delta' \leq \theta \delta$ the following hold, 
\begin{equation}\label{sm45}
\left\{
\begin{aligned}{}
\|f\|_{L^{\infty}(B^{+}_{1})}&\leq \epsilon\delta',\\
\|g\|_{L^{\infty}(B_1^0)}&\leq\epsilon\delta',\\
\text{osc}_{B^{+}_{1}}u&\leq\delta'
\end{aligned}
\right.
\end{equation}
then
\begin{equation}\label{smosc}
osc_{B^{+}_{\rho}}u\leq(1-\nu)\delta'.
\end{equation}
\end{thm}
\begin{proof}
We closely follow the ideas as in the proof of Proposition 2.2 in \cite{CF} with suitable  modifications  in our  situation. Let  $c_{1}$ be the constant from Lemma \ref{smesti}, when  the fully nonlinear operator $\tilde F$ under consideration is uniformly elliptic with ellipticity constants $\lambda,\Lambda$ in the region $p\in B_{\frac{\delta}{2}}$ instead of $ B_{\delta}.$ Also we fix $r_{0}$ sufficiently small so that Lemma \ref{smdecay} holds and then let $r_{1}=\frac{r_{0}}{16}.$ Let $\nu<\frac{1}{6}$ and $\mathfrak{N}$ be universal constants  to be chosen later such that additionally  the following is satisfied,

\begin{equation}\label{rt0}
\mathfrak{N}\nu<<1.
\end{equation}
Let us set 
\begin{equation}\label{defa}
a=\mathfrak{N}\nu\delta'\ \text{and}\ m=\inf_{B^{+}_{1}}u.
\end{equation} Suppose  that there exists $x_{0}\in B^{+}_{\frac{r_{0}}{2}}$ such that

\medskip

\emph{Assertion A:}

\begin{equation}\label{sm46}
u(x_{0})+\|g\|_{L^{\infty}(B_1^0)}-m<\frac{3}{2}\nu\delta'
\end{equation}
as well as 
\begin{equation}\label{sm47}
\sup_{B^{+}_{r_{1}}}u-\|g\|_{L^{\infty}(B_1^0)}-m>\frac{\delta'}{2}.
\end{equation}
We now make the following claim.

\medskip

\emph{Claim:}  The \emph{Assertion A} is false,  i.e. both the inequalities \eqref{sm46}, \eqref{sm47} cannot hold at the same time.

\medskip

Subsequently we show that this  leads to the validity of the oscillation decay as asserted  in \eqref{smosc} above.

\medskip

In order to prove the claim we assume on the contrary that both  the inequalities are correct and then derive  a  contradiction. 

Let us consider the following function
\begin{equation}
w=u-\|g\|_{L^{\infty}(B_1^0)}x_{n}.
\end{equation}
Then we note that $w$ satisfies the following differential inequality  in the viscosity sense
\begin{equation}\label{s48}
\left\{
\begin{aligned}{}
F_{1}(D^2w,Dw,x)&\leq f~\text{in}~B^{+}_{1},\\
w_{x_n}&\leq 0~\text{on}~B_1^0,\\
\end{aligned}
\right.
\end{equation}
where $F_{1}(M,p,x)=\tilde F(M,p+\|g\|_{L^{\infty}(B_1^0)}e_{n},x)$ and $e_{n}=(0,0,...,1).$ We have assumed that $\|g\|_{L^{\infty}(B_1^0)}\leq \epsilon\delta'$ so that  if we choose $\epsilon<\frac{\nu}{2}\leq \frac{1}{2},$ then we have that
\[\|g\|_{L^{\infty}(B_1^0)}\leq\frac{\delta'}{2}\leq\frac{\theta\delta}{2}\leq\frac{\delta}{2}~\hspace{0.5cm}~\big(\text{since}~~\theta\in(0,1)\big).\]
Consequently, $F_{1}$ is uniformly elliptic with the same ellipticity constant provided $p\in B_{\frac{\delta}{2}}.$ \\
Let us then consider the non-negative function 
\begin{equation}\label{defv}
v= u-m+(1-x_{n})\|g\|_{L^{\infty}(B_1^0)}.
\end{equation}
 It is easy to observe that $v$ satisfies \eqref{s48} in the viscosity sense because it differs from $w$ by a constant. We now let $\tilde{A}_{a}$ to be the set of points in $B^{+}_{1},$ where $v$  is bounded above by $a$ and can be touched by a paraboloid of opening $-a$ with vertex in $B^{+}_{1}.$

\medskip

\emph{Step 1:}
We first  show that given any $\eta>0$ sufficiently small depending on $r_1$, the following estimate holds 
\begin{equation}\label{sm54}
|B^{+}_{r_{0}}\cap \tilde{A}_{a}|>\frac{c_{0}}{2}|B^{+}_{r_{1}}\cap\{y~|~y_{n}>\eta\}|,
\end{equation}
with $c_0$ being independent of $\eta$.\\
In order to prove the claim, for every $y\in B^{+}_{r_{1}}\cap\{y~|~y_{n}>\eta\}$ let us consider the following paraboloid
\[P_{y}(x)=\frac{a}{2}\big[(r_{0}-r_{1})^{2}-|x-y|^{2}\big].\]
Since given $x$ for which   $|x|\geq r_{0},$ we have that  $|x-y|\geq|x|-|y|\geq r_{0}-r_{1},$ therefore $P_{y} (x)\leq0\leq v$~~for all $x\in\{z: 1>|z|\geq r_{0}\}\cap\{z_{n}>0\}.$\\
On the other hand, for all $x\in B^{+}_{\frac{r_{0}}{2}},$ we find that $|x-y|\leq |x|+|y|\leq \frac{r_{0}}{2}+r_{1}.$ Thus
\begin{equation}\label{sm48}
\left\{
\begin{aligned}{}
P_{y}(x)&\geq \frac{a}{2}\big[(r_{0}-r_{1})^{2}-\big(\frac{r_{0}}{2}+r_{1}\big)^{2}\big]\\
&=\frac{\mathfrak{N}\nu\delta'r^{2}_{0}}{2}\big[\big(\frac{15}{16}\big)^{2}-\big(\frac{9}{16}\big)^{2}\big]\\
&>\frac{3\nu\delta'}{2}\\
&\geq u(x_{0}) -m +(1-(x_{0})_{n})\|g\|_{L^{\infty}(B_1^0)}= v(x_0)\ (\text{by \eqref{sm46}}),
\end{aligned}
\right.
\end{equation}
where in the second line above, we have chosen $\mathfrak{N}$  sufficiently large so that  the third step in \eqref{sm48} above follows. Since \eqref{sm48} holds for $x\in B^{+}_{\frac{r_{0}}{2}},$ therefore, in particular, $P_{y}(x_{0})>\frac{3\nu\delta'}{2}.$ \\
Note also that $P_{y}(x)\leq a$~~for all $x,y\in B^{+}_{1}.$
Let us  now slide the paraboloids $P_{y}$ from below till it touches the function $v$ for the first time. Let $\tilde{A}$ denotes the set of contact points as $y$ varies in $B^{+}_{r_{1}}\cap\{y_{n}>\eta\}.$ Since the function $v$ satisfies \eqref{s48}, therefore  $P_{y}$ will not touch the function at any $\tilde{x}\in B_1^0.$  Otherwise by our choice of $y,$ we  would get
\begin{equation}\label{sm50}
\left\{
\begin{aligned}{}
0&\geq \partial_{x_n}(P_{y})(\tilde{x})~\hspace{0.5cm}\big(\text{because} ~v~\text{satisfies}~\eqref{s48}\big)\\
&=a(y-\tilde{x})_n\\
&=ay_{n}\geq a\eta>0~~\hspace{0.5cm}\big(\text{by the choice~of}~y\big),
\end{aligned}
\right.
\end{equation}
which is a contradiction. Therefore, in view of the above observations, we can infer that all contact points $\{\tilde x\}'s$ lie inside $B^{+}_{r_{0}}.$ Moreover  thanks to \eqref{sm48},   the following holds:
\begin{equation}\label{sm52}
\left\{
\begin{aligned}{}
0&>v(x_0)-\frac{3}{2}\nu\delta'\\
&\geq v(x_0)-P_{y}(x_{0})~~~\hspace{0.5cm}~\big(\text{by}~~\eqref{sm48}\big)\\
&\geq\min_{z\in B^{+}_{1}}\{v(z)-P_{y}(z)\}\\
&=v(\tilde x)-P_{y}(\tilde x )~\hspace{0.5cm}\big(\text{for a contact point}\ \tilde x \big)\\
&\geq v(\tilde x) -a~\hspace{0.5cm}~\big(\text{since}~~(P_{y}(x)\leq a)\big).
\end{aligned}
\right.
\end{equation}
This  implies that  $\tilde{A}\subset \tilde{A}_{a}\cap B^{+}_{r_{0}}.$ Thus by applying Lemma \ref{smesti} with $B=\overline{B^{+}_{r_{1}}} \cap \{z_{n} \geq\eta\},$ we obtain
\begin{equation}\label{sm53}
\left\{
\begin{aligned}{}
|B^{+}_{r_{0}}\cap \tilde{A}_{a}|&\geq|\tilde{A}|\\
&\geq c_{1}|B^{+}_{r_{1}}\cap\{y_{n}>\eta\}|-\frac{\|f\|^{n}_{L^{\infty}(B^{+}_{1})}}{a^{n}}|\tilde{A}|\\
&\geq c_{1}|B^{+}_{r_{1}}\cap\{y_{n}>\eta\}|-\frac{\epsilon^{n}}{\mathfrak{N}^{n}\nu^{n}}|\tilde{A}|\ (\text{using \eqref{sm45} and \eqref{defa}})\\
&\geq c_{1}|B^{+}_{r_{1}}\cap\{y_{n}>\eta\}|-\frac{\epsilon^{n}}{\mathfrak{N}^{n}\nu^{n}}|B^{+}_{r_{0}}\cap \tilde{A}_{a}|.
\end{aligned}
\right.
\end{equation}
Now, by choosing $\epsilon>0$~sufficiently small such that
\begin{equation}\label{small}
\frac{\epsilon^{n}}{\mathfrak{N}^{n}\nu^{n}} < \frac{1}{2},
\end{equation}
we obtain \eqref{sm54} with $c_0=\frac{c_1}{2}$. This finishes the proof of Step 1. \\
\emph{Step 2:} We  now show that there exists $\tilde \sigma \in (0,1)$ and $\tilde C > 0$ such that  the following estimate holds
\begin{equation}\label{s71}
|B^{+}_{r_{0}}\setminus \tilde{A}_{a\tilde{C}^{k_{0}}}|\leq (1-\tilde{\sigma})^{k_{0}}|B^{+}_{r_{0}}|,
\end{equation}
provided  $\tilde{C}^{k_{0}+1}a\leq \frac{\delta}{2}.$ From \eqref{sm54}, we find that
\begin{equation}
B^{+}_{r_{0}}\cap \tilde{A}_{a}\not=\emptyset.
\end{equation}
It is also clear that since the  sets $\tilde{A}_{a\tilde{C}^{k}}$ are increasing with respect to $k,$ therefore,
\begin{equation}
B^{+}_{r_{0}}\cap \tilde{A}_{a\tilde{C}^{k}}\not=\emptyset~~\text{for~all}~~k\in\N,
\end{equation}
where $\tilde{C}$ is  the constant as in Remark \ref{smr1} corresponding to $\delta/2$ instead of $\delta$.  Note that the hypothesis of  the  Lemmas \ref{smb}, \ref{smib} and \ref{smi} are satisfied with $\tilde C^k a$ instead of $a$  as long as $a \tilde C^{k+1} \leq \frac{\delta}{2}$. 

\medskip

Thus that for every $k\in\N,$ satisfying $a\tilde{C}^{k+1} \leq \delta/2$~ we can  apply Lemma \ref{smdecay} to the closed sets
\begin{equation}\label{sm552}
D_{0}=\overline{B^{+}_{r_{0}}}\cap \tilde{A}_{a\tilde{C}^{k}}~~\text{and}~~D_{1}=\overline{B^{+}_{r_{0}}}\cap\tilde{A}_{a\tilde{C}^{k+1}},
\end{equation}
to assert that 
\begin{equation}\label{sm56}
|B^{+}_{r_{0}}\setminus\tilde{A}_{a\tilde{C}^{k+1}}|\leq(1-\tilde{\sigma})|B^{+}_{r_{0}}\setminus\tilde{A}_{a\tilde{C}^{k}}|.
\end{equation}
Proceeding inductively,  we obtain
\begin{equation}\label{sm57}
|B^{+}_{r_{0}}\setminus\tilde{A}_{a\tilde{C}^{k}}|\leq (1-\tilde{\sigma})^{k}|B^{+}_{r_{0}}|,
\end{equation}
which completes the proof of Step 2.

\medskip

\emph{Step 3:}
We now define the following set
\begin{equation}\label{sm58}
E=\{x\in B^{+}_{r_{0}}~~|~~u(x)-m+(x_{n}-1)\|g\|_{L^{\infty}(B_1^0)}>\frac{\delta'}{4}\}.
\end{equation}
Then we claim that the following estimate holds for any $\eta>0$ sufficiently small, 
\begin{equation}\label{sm59}
|E|\geq\frac{c_1}{2}|B^{+}_{r_{1}}\cap\{y_{n}>\eta\}|,
\end{equation}

where $c_1$ is the constant from Lemma \ref{smesti}, when the operator under consideration is uniformly elliptic for $|p|<\delta/2$.\\
In order to prove \eqref{sm59}, for each $y\in B^{+}_{r_{1}}\cap\{y_{n}>\eta\},$ we consider the following paraboloid
\begin{equation}\label{sm60}
S_{y}(x)=\frac{\delta'}{(r_{0}-r_{1})^{2}}|x-y|^{2}+\frac{\delta'}{4}.
\end{equation}
By using the fact that $r_{1}=r_{0}/16,$ it is easy to observe that for all $x,y\in B^{+}_{r_{1}},$ we have
\begin{equation}\label{sm61}
S_{y}(x)\leq\frac{\delta'}{2}.
\end{equation}
Now using  \eqref{sm47}, we find
\begin{equation}\label{sm62}
\sup_{B^{+}_{r_{1}}} S_{y}(x)\leq \frac{\delta'}{2}<\sup_{B^{+}_{r_{1}}}u-\|g\|_{L^{\infty}(B_1^0)}-m\leq\sup_{B^{+}_{r_{1}}}\big(u+x_{n}\|g\|_{L^{\infty}(B_1^0)}\big)-\|g\|_{L^{\infty}(B_1^0)}-m.
\end{equation}
On the other hand for $x\in \{x~|~|x|\geq r_{0}\}\cap\{x_{n}>0\}$ since  $S_{y}(x)>\delta',$ therefore by \eqref{sm45}, we have
\begin{equation}\label{sm63}
\left\{
\begin{aligned}{}
S_{y}(x)&>\delta'\geq u(x)-m~\hspace{0.5cm}~\big(\text{by}~~\eqref{sm45}\ \text{and from the definition of $m$ as in }\ \eqref{defa} \big)\\
&\geq u(x)-m+(x_{n}-1)\|g\|_{L^{\infty}(B_1^0)}~~\hspace{0.5cm}~~\big(\text{since}~~(x_{n}-1)\|g\|_{L^{\infty}(B_1^0)}\leq0\big).
\end{aligned}
\right.
\end{equation}
Also for any $\bar{x}\in B_1^0$ and $y\in B^{+}_{r_{1}}\cap\{y_{n}>\eta\},$ we observe that
\begin{equation}\label{sm64}
\partial_{x_n}(S_{y})(\bar{x})=\frac{-2\delta'y_{n}}{(r_{0}-r_{1})^{2}} < 0.
\end{equation}
We now let  
\begin{equation}\label{defvt}
\tilde v=u+(x_{n}-1)\|g\|_{L^{\infty}(B_1^0)}-m.
\end{equation}
Then we observe that $\tilde v$ satisfies the following differential inequalities in the viscosity sense
\begin{equation}\label{sm65}
\left\{
\begin{aligned}{}
F_{2}(D^{2}\tilde v,D \tilde v,x)&\geq f~\text{in}~B^{+}_{1},\\
\tilde v_{x_n}&\geq 0~\text{on}~B_1^0,\\
\end{aligned}
\right.
\end{equation}
where $F_{2}(X,p,x)=\tilde F(X, p-\|g\|_{L^{\infty}(B_1^0)}e_{n},x),$ which is again uniformly elliptic as long as $p\in B_{\frac{\delta}{2}}.$\\
Now we slide the paraboloids $S_y$  from above until it touches the graph of $\tilde v$.   In view of \eqref{sm61},~\eqref{sm62}, \eqref{sm63} and \eqref{sm65}, all contact points lie inside $B^{+}_{r_{0}}.$ We denote by $K$ the set of all contact points as $y$ varies inside $B^{+}_{r_{1}}\cap\{y_{n}>\eta\}.$ We now apply  Lemma \ref{smesti}   from "above" to $\tilde v$, i.e. more precisely,   we apply that lemma to the function $-\tilde v$ which is touched from below by $-S_{y}(x)$. Note that in this case we have that $a=\frac{2\delta'}{(r_{0}-r_{1})^{2}}\leq \frac{2\theta\delta}{(r_{0}-r_{1})^{2}}$ since $\delta'\leq\theta\delta$. Therefore, if $\theta$ is chosen sufficiently small then we can ensure that  $0<a<\frac{\delta}{4}.$ We then observe  that  $-\tilde v$ satisfies the following inequalities
\begin{equation}\label{sm68}
\left\{
\begin{aligned}{}
G(D^{2}(-\tilde v),D(-\tilde v), x)&\leq -f~\text{in}~B^{+}_{1},\\
(- \tilde v)_{x_n}&\leq 0~\text{on}~B_1^0,\\
\end{aligned}
\right.
\end{equation}
in the viscosity sense, where $G(X, p, x)=-F_{2}(-X,-p)=-\tilde F(-X, -p-\|g\|_{L^{\infty}(B_1^0)}e_{n}, x),$ which is again uniformly elliptic for $p\in B_{\frac{\delta}{2}}.$ Therefore by applying  Lemma \ref{smesti}, we get
\begin{equation}\label{sm69}
\left\{
\begin{aligned}{}
|K|&\geq c_{1}|B^{+}_{r_{1}}\cap\{y_{n}>\eta\}|-\frac{\|f\|^{n}_{L^{\infty}(B^{+}_{1})}}{a^{n}}|K|\\
&\geq c_{1}|B^{+}_{r_{1}}\cap\{y_{n}>\eta\}|-|K|\frac{\epsilon^{n}}{\mathfrak{N}^{n}\nu^{n}}.
\end{aligned}
\right.
\end{equation}
At this point by using \eqref{small} we obtain the following estimate
\begin{equation}\label{sm70}
|K|\geq \frac{c_{1}}{2}|B^{+}_{r_{1}}\cap\{y_{n}>\eta\}|.
\end{equation}
Now we note that because of  \eqref{sm62}, at any contact point $x\in K,$ we have $\tilde v\geq\frac{\delta'}{4}$  and therefore $K\subset E.$ Consequently, we can assert that \eqref{sm59} holds. This completes the proof of Step 3.

\medskip

\emph{Step 4:} (\emph{Conclusion}.)

\medskip
Let $k_{0}\in \N$ be the largest integer  such that $\tilde{C}^{k_{0}+1}a\leq \frac{\delta'}{4}.$ Now  since $\delta'\leq \delta,$ so by using  the estimate \eqref{s71} in Step 2  we have
\begin{equation}\label{sm72}
|B^{+}_{r_{0}}\setminus \tilde{A}_{a\tilde{C}^{k_{0}}}|\leq (1-\tilde{\sigma})^{k_{0}}|B^{+}_{r_{0}}|.
\end{equation}

Now for  $x\in B^{+}_{1},$  we make the crucial observation that the  following inclusion holds:
\begin{equation}\label{sm71}
\left\{
\begin{aligned}{}
E&=\{x\in B^{+}_{r_{0}}~~| \tilde v(x)>\frac{\delta'}{4}\}\\
&\subset\{x\in B^{+}_{r_{0}}~~|~v(x)>\frac{\delta'}{4}\}\hspace{0.5cm}\big(\text{since}~ v\geq \tilde v\big)\\
&\subset\{x\in B^{+}_{r_{0}}~~|~v(x)>a\tilde{C}^{k_{0}}\}~\hspace{0.5cm}~\big(\text{since}~a\tilde{C}^{k_{0}}<\frac{\delta'}{4}\big)\\
&\subset B^{+}_{r_{0}}\setminus \tilde{A}_{a\tilde{C}^{k_{0}}}~\hspace{0.5cm}\big(\text{by~definition~of}~\tilde{A}_{a\tilde{C}^{k_{0}}}\big).
\end{aligned}
\right.
\end{equation}
Using \eqref{sm59},\eqref{sm72} and \eqref{sm71}, we have
\begin{equation}
\frac{c_1}{2}|B^{+}_{r_{1}}\cap\{y_{n}>\eta\}|\leq|E|\leq(1-\tilde{\sigma})^{k_{0}}|B^{+}_{r_{0}}|.
\end{equation}
Now letting $\eta \to 0,$ we obtain
\begin{equation}\label{v0}
\frac{c_1}{2}|B^{+}_{r_{1}}|\leq|E|\leq(1-\tilde{\sigma})^{k_{0}}|B^{+}_{r_{0}}|.
\end{equation}

Now note that    using $a= \mathfrak{N} \nu \delta'$, we have that 
\begin{equation}\label{k91}
k_{0}\sim |\log_{\tilde{C}}(\mathfrak{N}\nu)|.
\end{equation}  At this point we  first let   $\mathfrak{N}$ large enough so that all previous arguments apply.  Subsequently if   $\nu$ is chosen   small  enough,  then  thanks to \eqref{k91}, we have that   $k_0$ becomes too large    so that   \eqref{v0} is violated ( note that $r_1= \frac{r_0}{16}$). This leads to a contradiction.

 Note that we can accordingly  choose  $\epsilon$ sufficiently small such that \eqref{small} holds as well.

Therefore, we  finally  obtain that   for appropriately chosen $\mathfrak{N}, \nu,\ve$ as above,  either \eqref{sm46} or \eqref{sm47} fails. Suppose first that \eqref{sm46} fails. Then since $\|g\|_{L^{\infty}(B_1^0)}\leq\delta'\epsilon<\frac{\delta'\nu}{2}$ (by  our choice of $\epsilon$), therefore we have;
\[u(x)-m\geq\nu\delta'~~~\text{for~all}~~x\in B^{+}_{r_{1}},\]
where we also use the fact that  $r_1 < r_0/2$.
Consequently, \eqref{smosc} follows with $\rho=r_1$. Now, suppose  instead that  \eqref{sm47} fails.  Then  in this case we have that
\[\sup_{B^{+}_{r_{1}}}u-\|g\|_{L^{\infty}(B_1^0)}-m\leq\frac{\delta'}{2},\]
that is,
\[\sup_{B^{+}_{r_{1}}}u-m\leq\frac{2\delta'}{3},\]
since~$\|g\|_{L^{\infty}(B_1^0)})<\frac{\nu\delta'}{2}$~and~$\nu<1/3.$ Thus, \eqref{smosc} again follows in view of the fact that  $\frac{2}{3}<(1-\nu).$ This finishes the proof of the theorem.
\end{proof}

As a consequence of Theorem \ref{th1},  we also have the following rescaled boundary oscillation estimate whose proof is identical to that of Theorem 2.1 in \cite{CF}.

\begin{thm}\label{rescaled}
With $\tilde F,u,f,g$ as in Theorem \ref{th1}, we have that there exists universal $\nu, \kappa, \epsilon, \rho \in (0,1)$ such that if $\delta'>0$ and $k \in \mathbb{N}$ satisfy
\begin{equation}
\begin{cases}
\text{osc}_{B_1^{+}} u \leq \delta' \leq \rho^{k} \kappa \delta,
\\
||f||_{L^{\infty}(B_1^+)} \leq \ve \delta',\\
\ ||g||_{L^{\infty}(B_1^0)} \leq \ve \delta',
\end{cases}
\end{equation}
then
\begin{equation}\label{rs1}
\text{osc}_{B_{\rho^s}^+} u \leq (1-\nu)^s  \delta',\ \text{for $s=0, ..., k+1$}.
\end{equation}

\end{thm}

\section{Improvement of flatness and the proof of our main result}\label{end}
We now establish our main result Theorem \ref{main} using the non perturbative H\"older estimates proved in Sections \ref{sma} and \ref{lga}.  We first show how to reduce the considerations to flat boundary conditions.

\medskip

\subsection{Reduction to flat boundary conditions:}\label{flt1}
Since $\Omega \in C^{2}$,  we can flatten the boundary using coordinates which employs the distance function to the boundary $\partial \Omega$. See for instance Lemma 14.16 in \cite{GT} or the Appendix in \cite{BL}.  We crucially note that such coordinates  preserve the Neumann boundary conditions unlike standard flattening which changes Neumann conditions to oblique derivative conditions in general. Consequently, without loss of generality, we may  consider the following flat boundary value problem
\begin{equation}\label{deg1}
\begin{cases}
\langle A(x)Du, Du\rangle^{\beta/2} F(D^2u, Du, x) = f \ \text{in $B_1^{+}$},
\\
u_{x_n}=g\ \text{on $B_1^0$},
\end{cases}
\end{equation}
where $A$ is a uniformly elliptic positive definite matrix with Lipschitz coefficients.   Moreover  such a transformation ensures that the resulting $F$ is uniformly elliptic in $D^2u$ and Lipschitz in $Du$.  Without loss of generality, we will also assume that $\beta>0$ since the case $\beta=0$ is classical.

\medskip

\subsection{Improvement of flatness}

We  first state and prove a  compactness result for a perturbed variant of \eqref{deg1}. This  can be regarded as the boundary analogue of Lemma 4.2 in \cite{CF}.

\begin{lem}\label{comp1}
Let $u$  be such that   $|u| \leq 1$  and is   a viscosity solution to the following Neumann problem,
\begin{equation}\label{deg2}
\begin{cases}
\langle A(x)(Du+p), (Du+p)\rangle^{\beta/2} F(D^2u, Du, x) = f \ \text{in $B_1^{+}$},
\\
u_{x_n}=g\ \text{on $B_1^0$},
\end{cases}
\end{equation}
where $p \in \R^n$, $A$ is Lipschitz and uniformly elliptic and $F$ is uniformly elliptic in $M$ with ellipticity bounds $\lambda$ and $\Lambda$, Lipschitz in the gradient variable $q$ and continuous in $x$ with a modulus of continuity $\omega$. Also suppose $|F(0,0,0)| \leq 1$.  Furthermore, assume that $f \in C(\overline{B_1^+}),~~||f||_{L^{\infty}(B_1^+)} \leq 1$ and $g \in C^{\alpha_0}(B_1^0)$ with $||g||_{C^{\alpha_0}} \leq 1$.
Then given $\ve'>0$, there exists $L=L(\ve')>0$, such that if $|p| > L$, $|D_q F| \leq \frac{1}{L},$  then there exists $v \in C^{1, \alpha'}(\overline{B_{1/2}^+})$ for some $\alpha'$  universal (with  a universal $C^{1,\alpha'}$ estimate) 
such that
\begin{equation}\label{close}
||u-v||_{L^{\infty}(B_{1/2}^+)} \leq \ve'.
\end{equation}
\end{lem}
\begin{proof}
We first note that the equation  \eqref{deg2} can be rewritten as

\begin{equation}\label{deg4}
\begin{cases}
\langle A(x)(\frac{Du}{|p|}+e),\frac{Du}{|p|} + e \rangle^{\beta/2} F(D^2u, Du, x) = \frac{f}{|p|^{\beta}} \ \text{in $B_1^{+}$},
\\
u_{x_n}=g\ \text{on $B_1^0$},
\end{cases}
\end{equation}
where $e=\frac{p}{|p|}$. Therefore, we see that $u$ satisfies a uniformly elliptic PDE when $|Du| \leq \frac{|p|}{2}$.  Suppose on the contrary, the assertion is not true. Then there exist an $\ve_0>0$ and a sequence of  $u_k's, p_k's, F_k's, f_k's, g_k's$ with $|u_k| \leq 1, |f_k|\leq 1, ||g_k||_{C^{\alpha_0}} \leq 1,  |D_q F_k| \leq \frac{1}{k},  |p_k| > k$, such that  $F_k's$ have the same ellipticity bounds $\lambda, \Lambda$, are equicontinuous in $x$ with modulus $\omega$ and $u_k$ solves the following problem:
\begin{equation}\label{deg50}
\begin{cases}
\langle A(x)(\frac{Du_k}{|p_k|}+e_k),\frac{Du_k}{|p_k|} + e_k \rangle^{\beta/2} F_k(D^2u_k, Du_k, x) = \frac{f_k}{|p_k|^{\beta}} \ \text{in $B_1^{+}$},\ e_k=\frac{p_k}{|p_k|}
\\
(u_k)_{x_n}=g_k\ \text{on $B_1^0$}
\end{cases}
\end{equation}
and such that  $u_k$'s are not $\ve_0$ close to any $v\in C^{1,\alpha'}(\overline{B^{+}_{1/2}})$. We now rewrite the first equation in \eqref{deg50} as follows:
\begin{align}\label{deg13}
&\langle A(x)(\frac{Du_k}{|p_k|}+e_k),\frac{Du_k}{|p_k|} + e_k \rangle^{\beta/2} \Big( F_k(D^2u_k, Du_k, x) - F_k(0, Du_k, x)\Big)= \frac{f_k}{|p_k|^{\beta}}
\\
& -\langle A(x)(\frac{Du_k}{|p_k|}+e_k),\frac{Du_k}{|p_k|} + e_k \rangle^{\beta/2}F_k(0, Du_k, x) \ \text{in $B_1^{+}$}.
\notag
\end{align}
Now, notice that the operators in \eqref{deg13} above satisfy the structural assumptions SC1)- SC3) as in Section \ref{lga} and are uniformly elliptic for $|Du_k| \leq \frac{|p_k|}{2}$. Before proceeding further, we make the following important discursive remark.
\begin{rem}
Over here, the reader should note that the reason as to why we subtract off $F_k(0, Du_k, x)$ is to ensure that SC2) holds. Note that even if we start with $F$ satisfying SC2), after flattening such a condition is not  necessarily preserved. 
\end{rem}
Now similar to the proof of Lemma 4.2 in \cite{CF}, we look at the following rescaled functions
\[
w_k(x)= \theta_k (u_k(x)-u(0)),
\]
where
\begin{equation}\label{theta}
\theta_k=\max\big\{ \frac{1}{|p_k|}, ||D_q F_k||\big\}.
\end{equation}
Then, it follows that $w_k$ solves:
\begin{align}\label{deg14}
&\langle A(x)(\frac{Dw_k}{\theta_k |p_k|}+e_k),\frac{Dw_k}{\theta_k |p_k|} + e_k \rangle^{\beta/2} \theta_k( F_k(D^2w_k/\theta_k, Dw_k/\theta_k, x) - F_k(0, Dw_k/\theta_k, x))= \theta_k\frac{f_k}{|p_k|^{\beta}}
\\
& -  \theta_k \langle A(x)(\frac{Dw_k}{\theta_k |p_k|}+e_k),\frac{Dw_k}{\theta_k |p_k|} + e_k \rangle^{\beta/2}F_k(0, Dw_k/\theta_k, x) \ \text{in $B_1^{+}$}.
\notag
\end{align}
Moreover, $w_k$ satisfies in the viscosity sense the Neumann condition $(w_k)_{x_n} = \theta_k g_k$. Also from \eqref{deg14} it follows that  $w_k$  solves a degenerate elliptic problem  which is uniformly elliptic independent of $k$ when $|Dw_k| \leq 1/2= \delta$. Now let $\rho, \kappa, \ve, \nu$ be as in  Theorem \ref{rescaled} corresponding to $\delta=\frac{1}{2}$. In the region of uniform ellipticity it is easily seen that the scalar term
\begin{align}\label{st5}
& \tilde f_k = \theta_k\frac{f_k}{|p_k|^{\beta}}
\\
& -  \theta_k \langle A(x)(\frac{Dw_k}{\theta_k |p_k|}+e_k),\frac{Dw_k}{\theta_k |p_k|} + e_k \rangle^{\beta/2}F_k(0, Dw_k/\theta_k, x)
\notag
\end{align}
 satisfies $|\tilde f_k| \leq C_0\theta_k$. This follows from the expression of $\theta_k$ as in \eqref{theta}. Likewise, we have that $|\theta_k g_k| \leq  \theta_k$.  We now let $\delta'= \frac{ C_0 \theta_k}{\ve}.$  For a given $k,$ let $m_k$ be the largest integer such that
\[
\frac{C_0 \theta_k}{\ve}  \leq \rho^{m_k} \kappa \delta.
\]
Note that $m_k \to \infty$ as $k \to \infty$. Then it follows from the estimate in Theorem \ref{rescaled} that
\[
\text{osc}_{B_{\rho^{s}}^+} w_k \leq  C (1-\nu)^{s} \theta_k,\ s=1, ..., m_k.
\]
Scaling back to $u_k$ we obtain
\[
|u_k(x) - u_k(0)| \leq C|x|^{\alpha}\ \text{as long as $|x| \geq \rho^{m_k}$},
\]
where $\alpha= - \text{log}_{\rho}(1-\nu)$. Likewise one has a similar H\"older estimate at every boundary point in $B_{3/4}^0$.  The interior version of such estimates  follows from \cite{CF}.
This is enough to show that $\{u_k\}$;s are equicontinuous upto $\{x_n=0\}$ and consequently Arzela-Ascoli can be applied. Therefore, there  exists a subsequence which we still denote by $\{u_k\}$  which converges in $\overline{B_{3/4}^{+}}$ to some $v_0$. By passing to another subsequence, we can also assume  that $ F_k \to F_0$ which has the same ellipticity bounds and is independent of $q$ (since $D_qF_k \to 0$), $e_k \to e_0$ with $|e_0|=1$ and also $g_k \to g_0$ in $C^{\alpha_0}$. In a standard way, one can show that since $\frac{f_k}{|p_k|^{\beta}} \to 0$,  therefore  $v_0$ is a viscosity solution to

\begin{equation}\label{deg5}
\begin{cases}
\langle A(x)e_0, e_0 \rangle^{\beta/2} F_0(D^2v_0, x) = 0~~~\text{in}~B_{3/4}^{+},
\\
(v_0)_{x_n}=g_0\ \text{on $B_{3/4}^0$}.
\end{cases}
\end{equation}
For relevant stability results, we refer to Proposition 2.1 in \cite{LZ}. Now since  $\langle A(x)e_0, e_0 \rangle^{\beta/2} \ >0$, therefore, we can conclude that $v_0$ is a solution to
\begin{equation}\label{deg100}
\begin{cases}
 F_0(D^2v_0, x) = 0\ \text{in $B_{3/4}^{+}$},
\\
(v_0)_{x_n}=g_0\ \text{on $B_{3/4}^0$}.
\end{cases}
\end{equation}
Now, from the regularity results in \cite{MS}, it follows that $v_0 \in C^{1, \alpha'}(\overline{B_{1/2}^+})$ for some $\alpha'>0$ with universal bounds 
which immediately leads to a  contradiction for large enough $k's$.

\end{proof}
Before we state and prove  the improvement of flatness result for the perturbed equations  as in Lemma \ref{comp1}, we  first introduce a few universal parameters. Let 
\[
F(M, q, x): S(n) \times \R^n \times \R^n \longrightarrow \R,
\]
such that $F$ is uniformly elliptic in $M$ with ellipticity constants $\lambda, \Lambda$, Lipschitz in $q$ with Lipschitz bound say $1$  and continuous in $x$ with some modulus of continuity $\omega$. Also assume that $|F(0,0,0)| \leq 1$. Let $\alpha', C>0$ be  universal constants such that the following estimate holds 
\begin{equation}\label{universal}
||w||_{C^{1, \alpha'}(\overline{B_{1/2}^+})} \leq C,
\end{equation}
for any $w$ which is a   viscosity solution to the following problem:
\begin{equation}\label{de1}
\begin{cases}
F(D^2w, Dw, x) =0~\text{in}~~B_{3/4}^+,\\
|w|\leq 1,\\
w_{x_n}= g\ \text{on} ~~B_{3/4}^0~~~~ \text{and~~~~$ ||g||_{C^{\alpha_0}} \leq 1$ for~some~fixed $\alpha_0>0.$}
\end{cases}
\end{equation}
The existence of such $\alpha', C$ follows from the regularity results in \cite{MS}. We also note that from \eqref{universal}, the following estimate can be deduced,

\begin{equation}\label{universal1}
|w(x) - \tilde L(x)| \leq C|x|^{1+\alpha'},
\end{equation}
where $\tilde L$ is the affine approximation of $w$ at $0$. We now state the relevant improvement of flatness result when $|p|$ is large.
\begin{lem}\label{flat1}
With $u, A, f, g, p, F$ as in Lemma \ref{comp1}, there exist  universal $\ve_0>0$, $r \in (0,1 )$ and $\alpha>0$ such that if $|p| >L(\ve_0)$, $|D_q F| \leq \frac{1}{L(\ve_0)}$, then there exists an affine function $\tilde L$ with universal bounds as in \eqref{universal} such that
\begin{equation}\label{ft1}
||u- \tilde L||_{L^{\infty}(B_r^+)} \leq r^{1+\alpha}.
\end{equation}
\end{lem}

\begin{proof}
From Lemma \ref{comp1}, we have that  given $\ve'>0$, there exists $L(\ve')>0$ such that if $|p|>L(\ve'), |D_q F| \leq \frac{1}{L(\ve')}$, then there exists $v$ which is a solution to an equation of the type \eqref{de1} such that
\begin{equation}\label{l1}
||u-v||_{L^{\infty}(B_{1/2}^+)} \leq \ve'.
\end{equation}
Now from \eqref{universal1}, we have
\begin{equation}
|v(x) - \tilde L(x)| \leq C|x|^{1+\alpha'}.
\end{equation}
where $\tilde L$ is the affine approximation of $v$ at $0$.  We first choose 
\begin{equation}\label{choice}
\alpha< \min\big\{ \alpha_0, \alpha', \frac{1}{1+\beta}\}.
\end{equation}

 Subsequently  we  choose $r$ small enough such that
\begin{equation}\label{l2}
Cr^{1+\alpha'}\leq \frac{r^{\alpha+1}}{2},
\end{equation}
where $C, \alpha'$ are as in \eqref{universal1}. Finally we let  $\ve_0=\ve'= \frac{r^{\alpha+1}}{2}$. Therefore, the desired estimate in \eqref{ft1} follows from \eqref{l1}-\eqref{l2} by an application of  triangle inequality provided $|p| > L(\ve_0)$ and $|D_q F| \leq \frac{1}{L(\ve_0)}.$
\end{proof}
Before, proceeding further, we make the following important remark.
\begin{rem}
We  note that although in the proof of Lemma \ref{flat1}, one only needs to take $\alpha < \alpha',$ however for subsequent iterative arguments which involves rescaling, we have to additionally ensure that $\alpha<\text{min}\big\{\alpha_0, \frac{1}{1+\beta}\big\}.$
\end{rem}
We now have the analogous improvement of flatness result when $|p| \leq L(\ve_0).$
\begin{lem}\label{flat2}
Let $u$  such that $|u| \leq 1$  be a viscosity solution to \eqref{deg2}  where  $|p|\leq L(\ve_0)$ with $\ve_0$  as in Lemma \ref{flat1}.  Then there exists $\eta>0$ such that if
$||f||_{L^{\infty}}, ||g||_{C^{\alpha_0}} \leq \eta,$ then there exists an affine function $\tilde L= \tilde a +  <\tilde b , x>$ ( $, \tilde a \in \mathbb{R}, \tilde b \in \mathbb{R}^n$) with universal bounds such that
\begin{equation}\label{ft2}
||u- \tilde L||_{L^{\infty}(B_r^+)} \leq r^{1+\alpha},
\end{equation}
where $r, \alpha$ are as in Lemma \ref{flat1}. Moreover  we also additionally have that
\begin{equation}\label{en0}
<\tilde b, e_n>=0
\end{equation}
\end{lem}

\begin{proof}
\emph{Step 1:}. We first show that given $\ve>0$, there exists $\eta=\eta(\ve)>0,$ such that if $||f||_{L^{\infty}}, ||g||_{C^{\alpha}} \leq \eta,$ then there exists a function $v$ which solves
\begin{equation}\label{de5}
\begin{cases}
F(D^2v, Dv, x)=0\ \text{in $B_{3/4}^+$,}
\\
v_{x_n}= 0\ \text{on $B_{3/4}^0,$}
\end{cases}
\end{equation}
and
\[
||v-u||_{L^{\infty}(B_{1/2}^+)} \leq \ve.
\]
If not, then there exists $\ve>0$  for which the assertion is violated  for a sequence $u_k, f_k, g_k, p_k$ such that $f_k, g_k \to 0$, $|p_k| \leq L(\ve_0)$ and where  $u_k$ solves the following problem
\begin{equation}\label{de7}
\begin{cases}
\langle A(x)(Du_k+p_k), Du_k + p_k \rangle^{\beta/2} F(D^2u_k, Du_k, x) = f_k \ \text{in $B_1^{+},$}
\\
(u_k)_{x_n}=g_k\ \text{on $B_1^0.$}
\end{cases}
\end{equation}
Now, since $|p_k| \leq L(\ve_0),$ we find that  the equation is uniformly elliptic when $|Du_k| > 2 L(\ve_0)$( say in the viscosity sense). We also note  that \eqref{de7} can be  rewritten as:
\begin{equation}\label{de8}
\begin{cases}
 F(D^2u_k, Du_k, x) = \frac{f_k}{ \langle A(x)(Du_k+p_k), Du_k + p_k \rangle^{\beta/2}} \ \text{in $B_1^{+},$}
\\
(u_k)_{x_n}=g_k\ \text{on $B_1^0,$}
\end{cases}
\end{equation}
where for  $|Du_k| > 2L(\ve_0)$, one has
\[
\frac{|f_k|}{ \langle A(x)(Du_k+p_k), Du_k + p_k \rangle^{\beta/2}} \leq \frac{|f_k|}{\langle A(x) L(\ve_0), L(\ve_0)\rangle^{\beta/2}} \to 0 \ \text{as $k \to \infty.$}
\]
Consequently, from the uniform boundary  H\"older estimates as in Theorem \ref{holder1}, we have that  upto  a subsequence, $u_k \to v_0$ in $\overline{B_{3/4}^+}$, $p_k \to p_0$ such that $v_0$ is a viscosity solution to
\begin{equation}\label{de10}
\begin{cases}
\langle A(x)(Dv_0+p_0), Dv_0 + p_0 \rangle^{\beta/2} F(D^2v_0, Dv_0, x) = 0 \ \text{in $B_{3/4}^{+},$}
\\
(v_0)_{x_n}=0\ \text{on $B_{3/4}^0$}.
\end{cases}
\end{equation}
Such a stability result follows from an argument as in Proposition 2.1 in \cite{LZ}. Now, by arguing  as in the proof of  Lemma 6  in \cite{IS1}, we can assert that $v_0$ in fact solves
\begin{equation}\label{de}
\begin{cases}
F(D^2v_0, Dv_0, x) = 0 \ \text{in $B_{3/4}^{+},$}
\\
(v_0)_{x_n}=0\ \text{on $B_{3/4}^0.$}
\end{cases}
\end{equation}
This leads to a contradiction for large $k's.$

\medskip

\emph{Step 2:} (Conclusion)

\medskip

Now, we take $\eta$ corresponding to $\ve=\ve_0,$ where $\ve_0$ is as in  Lemma \ref{flat1}. The rest of the arguments are the same as in Lemma \ref{flat1} because the universal estimate   in \eqref{universal1} also holds for $v_0$.  Also \eqref{en0} follows because $\tilde L$ corresponds to  the affine approximation of $v_0$ at $0$  which  satisfies homogeneous Neumann condition as in \eqref{de}. 
\end{proof}
Now, we let
\begin{equation}\label{eta}
\eta_0= \min\Big\{ \eta,  1/L(\ve_0)\Big\},
\end{equation}
where $\eta, \ve_0$ are as in Lemma \ref{flat2} and $L(\ve_0)$ is as  in Lemma \ref{flat1} corresponding to $\ve_0.$
Finally as a consequence of Lemma \ref{flat1} and Lemma \ref{flat2}, we obtain that  the following uniform improvement of flatness which doesn't take into account the size of $|p|$. 
\begin{lem}\label{flat4}
Let $u$  be such that   $|u| \leq 1$  and is   a viscosity solution to the following Neumann problem,
\begin{equation}\label{deg2}
\begin{cases}
\langle A(x)(Du+p), (Du+p)\rangle^{\beta/2} F(D^2u, Du, x) = f \ \text{in $B_1^{+}$},
\\
u_{x_n}=g\ \text{on $B_1^0$},
\end{cases}
\end{equation}
where $p \in \R^n$, $A$ is Lipschitz and uniformly elliptic and $F$ is uniformly elliptic in $M$ with ellipticity bounds $\lambda$ and $\Lambda$, Lipschitz in the gradient variable $q$ and continuous in $x$ with a modulus of continuity $\omega$. Also suppose $|F(0,0,0)| \leq 1$.  Then with $\eta_0$ as in \eqref{eta} above, we have that  if  $||f||_{L^{\infty}}, ||g||_{C^{\alpha_0}}, |D_q F| \leq \eta_0$, then there exists an affine function $\tilde L= \tilde a+ <\tilde b, x>$ with universal bounds such that
\begin{equation}\label{ft0}
\begin{cases}
||u- \tilde L||_{L^{\infty}(B_r^{+})} \leq r^{1+\alpha},\\
<\tilde b, e_n>=0
\end{cases}
\end{equation}
where   $r, \alpha \in (0,1)$ are universal constants. Furthermore, we can additionally ensure that $\alpha$ satisfies \eqref{choice}.
\end{lem}

With Lemma \ref{flat4} in hand, we now prove our main result  Theorem \ref{main}.
\subsection*{Proof of Theorem \ref{main}}

\begin{proof}
\emph{Step 1:} (Basic reductions) In view of our discussion in subsection \ref{flt1}, we may first assume that $\partial \Omega = \{x_n=0\}$ and $u$ solves an equation of the type \eqref{deg1}.  Then, by letting
\[
u_s(x)= \frac{u(sx)}{s},
\]
we have that $u_s$ solves
\begin{equation}
\begin{cases}
\langle A(sx) Du_s(x), Du_s(x)\rangle^{\beta/2} sF(\frac{D^2u_s}{s}, Du_s, sx)= sf(sx).
\\
(u_s)_{x_n}(x)= g(sx)\ \text{on $\{x_n=0\}$}.
\end{cases}
\end{equation}
Now, by choosing $s$ sufficiently small, we can ensure that the operator
\[
F_s(M, q, x)= sF\bigg(\frac{M}{s},q, sx \bigg)
\]
satisfies $|D_q F_s| \leq \eta_0$ and also that 

\begin{equation}\label{half}
|F_s(0,0,0)| \leq \frac{1}{2}.
\end{equation}
  Subsequently we let  $u_s$ as our new $u$ and $F_s$ as our new $F$ which now additionally satisfies  $|D_q F| \leq \eta_0.$ Then  by letting $v= u - g(0)x_n- u(0),$ we have that $v(0)=0$  and it  solves
\begin{equation}\label{red1}
\begin{cases}
\langle A(x)(Dv + g(0)e_n), Dv+ g(0)e_n)\rangle^{\beta/2} F(D^2v,  Dv + g(0)e_n, x)= f~~\text{in}~~B^{+}_1,
\\
v_{x_n}= g - g(0)\ \text{on $B_1^0$},
\end{cases}
\end{equation}
We now define 
\[\tilde v= \frac{v}{\kappa}\]
 where \[\kappa=  \bigg(1+ ||v||_{L^{\infty}} + (\frac{||f||_{L^{\infty}}}{2 \eta_0})^{\frac{1}{1+\beta}} + \frac{||g||_{C^{\alpha_0}}}{2\eta_0} \bigg)\]
with $\eta_0$  as in Lemma \ref{flat4}. Then  we observe  that $\tilde v$ solves
\begin{equation}\label{red2}
\begin{cases}
 \langle A(x) (D\tilde v +\frac{g(0)}{\kappa}e_n), D\tilde v +\frac{g(0)}{\kappa}e_n\rangle ^{\beta/2} \kappa^{-1} F(\kappa D^2 \tilde v, \kappa D\tilde v, x)= \kappa^{-(1+\beta)} f(x)=\tilde f\ \text{in $B_1^+$},
\\
(\tilde v)_{x_n}= \tilde g= \frac{g}{\kappa}\ \text{on $B_1^0.$}
\end{cases}
\end{equation}
Now since $\kappa>1,$ we find that the new operator in \eqref{red2} satisfies similar structural conditions as $F.$ Moreover, we additionally have that $\|\tilde v\| \leq 1,$ $||\tilde f||_{L^{\infty}} \leq \eta_0, ||\tilde g||_{C^{\alpha_0}} \leq \eta_0.$ Thus by letting $\tilde v$ as our new $v,$ $\tilde g$ as our new $g$ and so on, we may assume without loss of generality that $v$ satisfies an equation of the type \eqref{deg2} for some $p$  and such that $ |D_q F|, \|f\|_{L^{\infty}}, ||g||_{C^{\alpha_0}} \leq \eta_0$  where $\eta_0$.  Moreover, we also have that  for our new $g$, $g(0)=0$.

\medskip
\emph{Step 2:}
We now show that for all $r, \alpha$ as in Lemma \ref{flat4}, we have that for every $k=0, 1, 2...,$ there exists $L_k= \langle b_k, x\rangle $ such that
\begin{equation}\label{as1}
\begin{cases}
||v- L_k||_{L^{\infty}(B_{r^{k}}^+)} \leq r^{k(1+\alpha)},
\\
\langle b_k, e_n\rangle=0,
\\
|b_k - b_{k+1}| \leq Cr^{k\alpha}.
\end{cases}
\end{equation}
We prove the claim in \eqref{as1} by induction. For $k=1,$ it follows from Lemma \ref{flat4} in view of  our reductions as in Step 1. Also note that since $v(0)=0$, by keeping track of the arguments that leads to Lemma \ref{flat4}, we can additionally ensure that $L_1(0)=0$.  We now assume that the assertion in \eqref{as1} holds upto some $k$. For such a $k$, we let
\[
w= \frac{(v- L_k)(r^k x)}{r^{k(1+\alpha)}}.
\]
Then, we have that $|w| \leq 1$ in $B_1^+$ and it satisfies the following inequalities in the viscosity sense
\begin{equation}\label{red4}
\begin{cases}
\langle A(r^k x) (Dw + p_k), Dw + p_k\rangle^{\beta/2}  r^{k(1-\alpha)}F( r^{k(\alpha-1)} D^2 w, r^{k\alpha}Dw + b_k, r^k x)= f_k(x)\  \text{in $B_1^+$}
\\
(w)_{x_n}= g_k(x) \ \text{on $T_1,$}
\end{cases}
\end{equation}
where $p_k= r^{-k\alpha} p +r^{-k\alpha} b_k$, $f_k(x)= r^{k(1- \alpha(1+\beta))} f(r^k x)$ and $g_k(x)= r^{-k\alpha}  g(r^k x)$. Now, since $||g||_{C^{\alpha_0}} \leq \eta_0, g(0)=0$ and $\alpha_0 > \alpha$, therefore, one can deduce easily  that $||g_k||_{C^{\alpha_0}} \leq \eta_0$. Also since $\alpha < \frac{1}{1+\beta}$ and $||f||_{L^{\infty}} \leq \eta_0$, therefore we can infer  that $||f_k||_{L^{\infty}} \leq \eta_0$.  Moreover, it  also follows that the  operator $F_{r, k}$
 in \eqref{red4} defined as
\[
F_{r, k}(M, q, x)=r^{k(1-\alpha)} F(r^{k(\alpha-1)} M, r^{k\alpha} q + b_k, r^k x),
\]
has the same ellipticity bounds as $F.$ Moreover, $||D_q F_{r, k}||  \leq r^{k} ||D_q F|| \leq \eta_0$ since $r<1$. Also using \eqref{half} we have that
\begin{equation}
|F_{r,k}(0,0,0)| \leq r^{k(1-\alpha)} |F(0,0,0)| + \eta_0 r^{k(1-\alpha)} |b_k| \leq  \frac{1}{2} + C \eta_0 \leq 1
\end{equation}
provided $\eta_0$ is further  adjusted in the beginning. 

Therefore, we can again apply Lemma \ref{flat4}  to obain for some  $\tilde L(x)= <\tilde b, x>$ satisfying $<\tilde b,e_n>=0$ that the following inequality holds, 
\[
||w- <\tilde b, x>||_{L^{\infty}(B_r^+)}  \leq r^{1+\alpha}
\]
Over here, we crucially used the fact that since $w(0)=0$, therefore as for $k=1$, we also additionally obtain that $\tilde L(0) = \tilde a=0$  by applying Lemma \ref{flat4} in  this specific  situation. 
Scaling back to $v,$ we deduce that \eqref{as1} holds for $k+1$ with $L_{k+1}(x)= L_k (x) + r^{k(1+\alpha)} \tilde L(r^{-k} x).$  This verifies the induction step and finishes the proof of Step 2.

\medskip

\emph{Step 3} (Conclusion)

\medskip

 It follows from \eqref{as1} by a standard analysis  argument that $L_0= \text{lim}_{k \to \infty}\ L_k$ is the affine approximation of order $1+\alpha$ at $0$ for $v$ and consequently $L_0 + g(0) x_n$ is the $1+\alpha$ order affine approximation for $u$ at $0.$ Likewise we have an affine approximation of order $1+\alpha$ at all boundary points.  Now going back to the original domain $\Omega,$  we can assert that there exists an affine approximation for $u$ of order $1+ \alpha$ at all points of $\partial \Omega \cap B_1.$ At this point, by a standard argument as in  \cite{MS}, one can combine the boundary $C^{1+\alpha}$ estimate with the interior ones as in \cite{IS1} to conclude  that $u \in C^{1, \alpha}(\overline{\Omega \cap B_{1/2}})$. Over here we note that although the interior regularity  result in \cite{IS1} is stated for
\[
|Du|^{\beta} F(D^2u) =f
\]
nevertheless, the proof works exactly the same way for equations of the type

\[
|Du+L|^{\beta} F(D^2u, x)=f,
\]
when $F$ depends continuously on $x$.  This finishes the proof of the theorem.

\end{proof}

\begin{proof}[ Proof of Corollary \ref{main1}]

We first rewrite the boundary condition in Corollary \ref{main1} as follows
\begin{equation}
\begin{cases}
|Du|^{\beta} F(D^2u, x) = f, \ \text{in $\Omega \cap B_1(0)$, $0 \in \partial \Omega$, $\beta \geq 0$},
\\
u_{\nu} =\tilde g,\ \text{on $\partial \Omega \cap B_1(0)$},
\end{cases}
\end{equation}
where $\tilde g= g - h(x) u$. Then by flattening and  by  applying the H\"older regularity result Theorem \ref{holder1}, we obtain that $u$ is $C^{\alpha}$ upto the boundary. This in turn implies that  $\tilde g$  is H\"older continuous and consequently the conclusion follows from Theorem \ref{main}. 

\end{proof}

In closing, we make the following remark.

\begin{rem}
It seems plausible that the techniques in this paper can be modified to yield $C^{1, \alpha}$  regularity  results for Neumann boundary problems of the type
\begin{equation}\label{plap}
\begin{cases}
|Du|^{\beta} (\Delta u + (p-2) \Delta_{\infty}^{N} u) = f, \ \text{in $\Omega \cap B_1(0)$, $0 \in \partial \Omega$, $\beta \geq 0$},
\\
u_{\nu} =\tilde g,\ \text{on $\partial \Omega \cap B_1(0)$},
\end{cases}
\end{equation}
where $\Delta_{\infty}^N u$ is the normalized infinity laplacian operator. The case when $\beta=0$ corresponds to the Poisson problem for the  normalized $p-$laplacian operator and this has been studied in various contexts in a  number of papers. See for instance \cite{APR}, \cite{Be}, \cite{BK} and one can find the references therein.  For general $\beta>0$, we refer to  \cite{AR} for the interior $C^{1, \alpha}$  regularity result for such equations and also to  \cite{IJS} and \cite{A} for the parabolic counterpart of such results. 
\end{rem}

\end{document}